\documentclass[a4paper, 10pt, twoside, notitlepage]{amsart}
\usepackage[utf8]{inputenc}
\usepackage{color}
\usepackage{amsmath} 
\usepackage{amssymb} 
\usepackage{amsthm}
\usepackage{esint}
\usepackage[colorlinks=true,linkcolor=blue]{hyperref}
\usepackage[export]{adjustbox}
\usepackage{thmtools}

\usepackage{enumitem}

\theoremstyle{plain}
\newtheorem{thm}{Theorem}
\newtheorem{prop}{Proposition}[section]

\newtheorem{lem}[prop]{Lemma}

\newtheorem{rmk}[prop]{Remark}

\newtheorem{alg}[prop]{Algorithm}

\newcommand {\R} {\mathbb{R}} \newcommand {\Z} {\mathbb{Z}}
 \newcommand {\N} {\mathbb{N}}
 
\newcommand {\p} {\partial}

\newcommand {\supp} {\text{supp}}

\DeclareMathOperator{\F}{\mathcal{F}}

\DeclareMathOperator {\dist} {dist}

\DeclareMathOperator {\intconv} {intconv}
\DeclareMathOperator {\conv}{conv}
\DeclareMathOperator {\inte} {int}

\DeclareMathOperator {\Cover} {Cover}
\DeclareMathOperator {\Replace} {Replace}

\begin{document}

\title[Flexibility, Rigidity and Some Numerical Implementations]{Convex Integration Arising in the Modelling of Shape-Memory Alloys: Some Remarks on Rigidity, Flexibility and Some Numerical Implementations}

\author{Angkana R\"uland }
\author{Jamie M Taylor}
\author{Christian Zillinger}

\address{
Max-Planck-Institute for Mathematics in the Sciences, Inselstraße 22, 04103 Leipzig, Germany}
\email{rueland@mis.mpg.de}

\address{
Department of Mathematics,
University of Southern California,
3620 S. Vermont Avenue,
Los Angeles, CA 90089-2532, US}
\email{zillinge@usc.edu}

\address{
Kent State University - Kent Campus, Department of Mathematical Sciences, Mathematics and Computer Science Building 233, Summit Street, Kent OH 44242, US}
\email{jtayl139@kent.edu}

\begin{abstract}
We study convex integration solutions in the context of the modelling of shape-memory alloys. The purpose of the article is two-fold, treating both rigidity and flexibility properties: Firstly, we relate the maximal regularity of convex integration solutions to the presence of lower bounds in variational models with surface energy. Hence, variational models with surface energy could be viewed as a \emph{selection mechanism} allowing for or excluding convex integration solutions. Secondly, we present the first numerical implementations of convex integration schemes for the model problem of the geometrically linearised two-dimensional hexagonal-to-rhombic phase transformation. We discuss and compare the two algorithms from \cite{RZZ16} and \cite{RZZ17}.
\end{abstract}

\maketitle

\section{Introduction}
\label{sec:intro}

Shape-memory alloys are materials undergoing a first order, diffusionless solid-solid phase transformation in which symmetry is lost. Here typically the high temperature phase, the \emph{austenite}, is highly symmetric (e.g. having atomistically a highly symmetric unit cell structure), while the low temperature phase, the \emph{martensite}, loses some of this symmetry. This loss of symmetry amounts to the presence of multiple, energetically equivalent \emph{variants of martensite}. This renders these materials rather complex and gives rise to many complicated, interesting microstructures in these materials \cite{B}.

\subsection{The Ball-James variational model}
\label{sec:BJ}
Mathematically, various features of the microstructures in shape-memory alloys have been quite successfully modelled by (static) energy minimization \cite{B3}, \cite{Kh13}:
\begin{align}
\label{eq:min}
\min\limits_{y(x)=M(x) \text{ on } \partial \Omega} \int\limits_{\Omega} W(\nabla y, \theta) dx .
\end{align}
Here $\Omega \subset \R^3$ denotes the \emph{reference configuration} (which is often chosen to be the material configuration in the austenite phase at fixed temperature), $y:\Omega \rightarrow \R^3$ denotes the \emph{deformation} of the material which is subjected to the boundary conditions $y(x)=M(x)$ on $\partial \Omega$.
The function $\theta:\Omega \rightarrow \R_+$ represents \emph{temperature} (in normalized units) and $W: \R^{3\times 3}_+ \rightarrow \R_+$ the \emph{stored energy function}. Physical requirements on $W$ are:
\begin{itemize}
\item[(i)] \emph{frame indifference}: $W(QM,\theta)=W(M,\theta)$ for each $M\in \R^{3\times 3}_+$, $Q \in SO(3)$ and $\theta \in \R_+$.
\item[(ii)] \emph{material symmetry}: $W(MH, \theta)= W(M,\theta)$ for each $M\in \R^{3\times 3}_+$, $\theta \in \R_+$ and $H \in \mathcal{P}$, where $\mathcal{P} \subset SO(3)$ denotes the point group of the material, which encodes the symmetries of the material. 
\end{itemize}
The symmetry of the phase transformation in conjunction with these two conditions then allows one to describe the \emph{energy wells} of $W$ for a martensitic phase transformation
\begin{align}
\label{eq:wells}
W(M,\theta)=0 
\Leftrightarrow M \in K(\theta)= \left\{
\begin{array}{ll}
\alpha(\theta)SO(3) \mbox{ if } \theta>\theta_c,\\
\bigcup\limits_{j=1}^{m} SO(3)U_j(\theta) \mbox{ if } \theta\leq \theta_c,
\end{array}
\right.
\mbox{ with } U_j(\theta)^t= U_j(\theta).
\end{align}
Here $\theta_c \in \R_+$ denotes the \emph{transformation temperature} between the austenite and the martensite, $\alpha(\theta)SO(3)$ with $\alpha: \R_+ \rightarrow \R_+$ models the austenite phase at temperature $\theta$ with its \emph{thermal expansion} $\alpha(\theta)$, and the matrices $SO(3)U_j(\theta)$ model the (energetically) equivalent variants of martensite at temperature $\theta\in \R_+$ \cite{Ball:ESOMAT}.
The two conditions (i), (ii) render the minimization problem \eqref{eq:min} \emph{highly non-(quasi)convex} and hard to analyse mathematically (even at a fixed temperature, which is the usual mathematical set-up) \cite{B1, B2}. 

\subsection{The $m$-well problem}
\label{sec:m_well}
As a consequence of the difficulty of dealing with the full problem \eqref{eq:min}, if one is interested in studying low energy microstructures, it is often useful to investigate \emph{exactly stress-free} deformations and microstructures. These are the solutions to the differential inclusion
\begin{align}
\label{eq:m_well}
\nabla y \in K(\theta) \mbox{ a.e. in } \Omega , \ y(x) = M(x) \mbox{ on } \partial \Omega,
\end{align}
where $K(\theta)$ is the set from \eqref{eq:wells}. If $\theta<\theta_c$, \eqref{eq:m_well} is also called an \emph{$m$-well problem} (due to the presence of $m$ variants of martensite). A natural question here is the \emph{existence} of solutions to \eqref{eq:m_well}, respectively, the existence of \emph{energy zero} solutions to \eqref{eq:min}. While highly oscillatory \emph{minimizing sequences} with limiting energy zero had been shown to exist for a large class of boundary data (related to the quasiconvex hull of $K(\theta)$; in general these sequences however only have measure valued limits), \cite{B3, M1}, the existence of exact energy zero solutions was only resolved later by relying on the technique of \emph{convex integration} \cite{MS, MS1, DaM12, K1, DKMS, G}. These quite surprising solutions however are rather ``wild" and oscillatory in that they use a cascade of different scales (similar to convex integration solutions in fluid mechanics and geometry, c.f. \cite{S} and the references therein). This ``wildness" can be quantified in terms of the regularity of the solutions to \eqref{eq:m_well}. In particular, it has been shown in different models \cite{DM1, DM2, K, DKMS, K, R16a} that convex integration solutions can only exist at rough regularities, giving rise to a dichotomy. For instance for the two-well problem \cite{DM1, DM2} (where there are exactly two rank-one connections between the wells) the following is known: 
\begin{itemize}
\item[(a)] If a solution is $BV$ regular, i.e. if $\nabla y \in BV$, then $\nabla y$ is a simple laminate, i.e. it is (up to boundary effects) ``one-dimensional" with a directional dependence which is determined by the structure of the set $K(\theta)$ from \eqref{eq:wells}.
\item[(b)] For any fixed $\theta<\theta_c$ and $M \in \inte (K(\theta)^{lc})$, where $K(\theta)^{lc}$ denotes the \emph{lamination convex hull} of $K(\theta)$, there exist solutions to \eqref{eq:m_well}. By default these solutions are $L^{\infty}$ regular, i.e. $\nabla y \in L^{\infty}$.
\end{itemize}
Thus, a natural question is whether there is a threshold between these two regimes in terms of a \emph{critical regularity threshold} (similar to the case in fluid mechanics and geometry \cite{S, I16, BDLSV17, CDS12}). In contrast to the situation in fluid mechanics, the problem \eqref{eq:m_well} does not have a natural scaling, which would help to indicate a critical threshold. In previous work together with B. Zwicknagl \cite{RZZ16, RZZ17}, we have hence studied the possibility of deriving higher regularity for the solutions from (b) for simplified ``linearised" models. It is one of the purposes of this article to show that \emph{scaling laws} yield lower bounds on the possible regularity of these solutions (and could be viewed as ``viscosity approximations" of the limiting model). Moreover, we present numerical simulations of convex integration solutions for the simplest possible model setting. We explain this in more detail in the following.

\subsection{Surface energy regularisation and scaling laws}
While various important \emph{qualitative} properties, e.g. the emergence of preferred twinning directions as well as the presence of clear volume fractions of certain variants of martensite, have been successfully explained by \eqref{eq:min} and \eqref{eq:m_well}, more \emph{quantitative} properties such as the emergence of length scales are not captured by this. In the literature (c.f. for instance \cite{KM1, CO, CO1, R16b, C1, CC15, KO}) this has been remedied by adding a higher order \emph{surface energy} contribution to \eqref{eq:min}, which is weighted with a small (material dependent) prefactor $\epsilon\in \R_+$:
\begin{align}
\label{eq:min_surf}
\min\limits_{y(x)=M(x) \text{ on } \partial \Omega} \left( \int\limits_{\Omega} W(\nabla y, \theta) dx + \epsilon^2 \int\limits_{\Omega}|\nabla^2 y|^2 dx \right).
\end{align}
From a mathematical point of view the presence of the higher order term in \eqref{eq:min_surf} regularises the problem and restores compactness to it. In general, the exact physical form of the higher order term in \eqref{eq:min_surf} is not known and various types of different regularisations have been used in the literature, including the diffuse, $L^2$ variant stated in \eqref{eq:min_surf}, but also the sharp, $BV$ variant (c.f. \cite{BMC09} for an overview on different models for this). In a sense \eqref{eq:min_surf} can be viewed as a \emph{selection criterion} for certain microstructures that arise in \eqref{eq:min} (and \eqref{eq:m_well}); the additional surface term in \eqref{eq:min_surf} distinguishing between wild, rough and highly oscillatory (approximate) solutions in \eqref{eq:min} and (approximate) solutions to \eqref{eq:min}, which are more regular (and hence use less surface energy). The properties of minimisers are often analysed by studying scaling limits of the functional \eqref{eq:min_surf}, c.f. for instance \cite{C1, KM1, CO, CO1, R16b, C1, CC15, KO}. While not predicting the exact form of the selected minimisers in \eqref{eq:min}, scaling does yield important information on these.
Recently, a more precise limiting analysis has been carried out in \cite{TS}. In spite of the various impressive advances in this direction, many points remain open in this context.

\subsection{The main results}
The purpose of this note is two-fold and addresses both rigidity and flexibility properties of the underlying differential inclusions: Firstly, we show that scaling laws complement the analysis of convex integration solutions in a precise sense, in that the presence of a scaling law (or rather a lower bound on the energy \eqref{eq:min_surf}) yields an upper bound on the maximal regularity of convex integration solutions (which might be well-known to experts in the field, but which we could not find in the literature). Here we exploit that, while \eqref{eq:m_well} itself does not have a prescribed scale predicting a possibly critical regularity, the model \eqref{eq:min_surf} has.

\begin{thm}
\label{thm:lower_bd}
Let $\Omega \subset \R^n$ be an open bounded $C^{1,1}$ domain. Let $E_{\epsilon}$ be the energy
\begin{align}
\label{eq:Eeps}
E_{\epsilon}
= \min\limits_{\nabla u = M \text{ a.e. in } \R^{n} \setminus \overline{\Omega}} \left\{ \int\limits_{\Omega}\dist^2(\nabla u, K) dx 
+ \epsilon^2 \int\limits_{\Omega}|\nabla^2 u|^2 dx  \right\},
\end{align}
where $K=K(\theta)$ for some fixed $\theta \in \R_+$ is as in \eqref{eq:wells}.
Assume that there exist constants $C>1$ and $\mu \in (0,\frac{1}{2})$ such that for all $\epsilon \in (0,\epsilon_0)$
\begin{align*}
  \epsilon^{2\mu} \leq C E_{\epsilon}.
\end{align*} 
Suppose that $u$ is a solution to 
\begin{align}
\label{eq:conv_int}
\begin{split}
\nabla u &\in K \mbox{ a.e. in } \Omega,\\
\nabla u &= M \mbox{ a.e. in } \R^n \setminus \overline{\Omega}.
\end{split}
\end{align}
If $v(x):=u(x)-Mx-b \in H^{1+s}(\R^n)$ for some $b\in \R^{n}$, $s\in \R$ with $\supp(v)\subset \overline{\Omega}$ and $\nabla v \in L^{\infty}(\R^n)$, then $s\leq \mu$.
\end{thm}
Loosely speaking this implies that for $\nabla u$ on an $H^{s}_{loc}(\R^n)$ scale, any convex integration solution $\nabla u \in L^{\infty}(\Omega)$ can be at most $H^{\mu}_{loc}(\R^n)$ regular.

We remark that the bound $\mu \in (0,\frac{1}{2})$ is related to the presence of trace estimates (c.f. the proof of Theorem \ref{thm:lower_bd} below).

Similar results can be obtained by using different forms of surface energies (c.f. Propositions \ref{prop:scaling_BV} and \ref{prop:scaling_p} in Section \ref{sec:rigidity_and_scaling}) and the corresponding Sobolev and Besov spaces. \\

Secondly, in addition to providing this upper bound on the maximal regularity of convex integration solutions in terms of the behaviour of scaling laws, we present numerical implementations of the convex integration schemes discussed in \cite{RZZ16} and \cite{RZZ17} in the case of the model setting of the geometrically linearised hexagonal-to-rhombic phase transformation (c.f. Section \ref{sec:hex_rhombic} for a more detailed description of this phase transformation). 

\subsection{Outline of the article}
The remainder of the article is organised as follows: First, in Section \ref{sec:rigidity_and_scaling}, we address the relation between scaling laws and rigidity results. In particular, we present the proof of Theorem \ref{thm:lower_bd} (and discuss variants of it). Next, in Section \ref{sec:flex}, we recall the convex integration algorithms from \cite{RZZ16} and \cite{RZZ17} for the model case of the hexagonal-to-rhombic phase transformation. Numerical implementations of these are presented in Section \ref{sec:numerics}, where we also compare the main features of the two algorithms. Finally, in Section \ref{sec:discuss} we summarise our findings and discuss interesting questions and models related to these. In the Appendix, we provide a proof of a fractional Poincar\'e inequality in the form in which we use it in Section \ref{sec:rigidity_and_scaling}.

\subsection*{Acknowledgements} 
The research of Jamie M Taylor leading
to these results has received funding from the European Research Council under the European Union’s Seventh Framework Programme (FP7/2007-2013)/ERC grant agreement no 291053. Christian Zillinger acknowledges the support of an AMS-Simons Travel Grant.

\section{Rigidity and Scaling Laws}
\label{sec:rigidity_and_scaling}

In this section we present the derivation of the bound on the maximal possible regularity of convex integration solutions. Here we do not yet specify the set-up to that of a particular phase transformation but remain in the general framework, which was layed out in Section \ref{sec:intro}.

We explain the close relation between scaling laws and the possible regularity of convex integration solutions (which is reminiscent of viscosity approximations, c.f. \cite{F95,S} for these ideas in the context of fluid mechanics):
Assume that for a bounded Lipschitz domain $\Omega \subset \R^n$ a function $u \in W^{1,\infty}(\Omega)$, which in addition satisfies $ \nabla u \in H^{\gamma}_{loc}(\R^n)$, is a solution to the differential inclusion
\begin{align*}
\nabla u &\in K \mbox{ a.e. in } \Omega,\\
\nabla u &= M \mbox{ a.e. in } \R^n \setminus \overline{\Omega},
\end{align*}
where $K:=K(\theta)$ denotes the wells from \eqref{eq:wells} for some fixed temperature $\theta>0$ (this corresponds to the case of affine boundary data in \eqref{eq:m_well}, which is already an interesting case). 
Then, any sequence $\nabla u_k$ in $H^{\gamma}_{loc}(\R^n)$ with $\nabla u_k \rightarrow \nabla u$ in $L^{2}_{loc}(\R^n)$ and with $\nabla u_k = M$ a.e. in $\R^n \setminus \overline{\Omega}$ is an admissible competitor in the minimization problem for $E_{\epsilon}$, where $E_{\epsilon}$ is the energy from \eqref{eq:Eeps}. The lower bound in the scaling law from Theorem \ref{thm:lower_bd}, then yields a bound on the maximal regularity of $\nabla u$.

The special role of the space $H^{\mu}$ (and the $W^{s,p}$ Sobolev  exponent $ps= 2 \mu$) can be guessed if we suppose that the scaling law $E_{\epsilon}(\nabla u) \sim \epsilon^{2\mu}$ holds. Assuming an equipartition of energy and treating the first term in \eqref{eq:Eeps} heuristically as an $L^2$ energy, we obtain that for a minimiser $u_{\epsilon}$ with fixed $\epsilon>0$
\begin{align*}
\|\nabla u_{\epsilon}\|_{L^2}^2 &\sim \epsilon^{2\mu}, \ \
\| \nabla^2 u_{\epsilon} \|_{L^2}^2 \sim \epsilon^{2\mu-2}. 
\end{align*}
Hence if there is a critical function space for $\nabla u$, it should be determined by the requirement that its scaling behaviour is independent of $\epsilon>0$. By interpolation, we thus guess that it is given by $\nabla u \in [ L^{2}, H^{1}]_{\mu} = H^{\mu}$.\\
Similar arguments indicate ``critical" function spaces if the energies in \eqref{eq:Eeps} are not diffuse, but for instance sharp interface models.

This intuition can be made rigorous and directly yields the proof of Theorem \ref{thm:lower_bd} (c.f. Section \ref{sec:thm1}).

\subsection{Function spaces}
\label{sec:FS}

Before proceeding with the discussion of the relation between regularity and scaling, for reference we first collect the function spaces, which we will be using in the sequel.
Denoting the Fourier transform by 
$\F w (k) = \int\limits_{\R^n} e^{i k \cdot x}w(x) dx$, for $s\in \R$ we consider the \emph{$L^2$ based fractional Sobolev spaces}
\begin{align*}
H^{s}(\R^n):=\{u\in \mathcal{D}'(\R^n): \ \|u\|_{H^{s}(\R^n)}:= \|\F^{-1}[(|\xi|^{2} +1)^{s/2}\F u]\|_{L^2(\R^n)}<\infty\}.
\end{align*}

Moreover, we will also use \emph{Besov spaces}. To this end we work with a \emph{Littlewood-Paley decomposition}. Following \cite[Appendix A]{Tao06}, we let $\varphi$ denote a non-negative, radially symmetric bump function supported in $\{k \in \R^d: \ |k|\leq 2\}$, which is equal to one on $\{k \in \R^d: \ |k|\leq 1\}$. Let $N \in 2^{\Z}$, i.e. assume that there exists $j\in \Z$ such that $N=2^{j}$. We then define the \emph{Littlewood-Paley projectors} $P_{<N}, P_{\geq N}, P_N$ to be 
\begin{align}
\label{eq:LP}
\begin{split}
\F(P_{< N} f)(k)&:= \varphi(k/N)\F f(k),\\
\F(P_{\geq N} f)(k)&:= (1-\varphi(k/N))\F f(k),\\
\F(P_{N} f)(k)&:= (\varphi(k/N)- \varphi(2k/N))\F f(k).
\end{split}
\end{align}
With this at hand, we
recall that for $s\in \R$ and $p\in [1,\infty)$ the space of \emph{Besov functions} $f \in B^{s}_{p,p}(\R^n)$ is defined as
\begin{align*}
B^{s}_{p,p}(\R^n)=\{f\in \mathcal{D}'(\R^n): \ \|f\|_{B^s_{p,p}(\R^n)}<\infty \},
\end{align*}
where
\begin{align*}
\|f\|_{B^{s}_{p,p}(\R^n)}:= \left( \sum\limits_{N\in 2^{\Z}} N^{sp} \|P_N f\|_{L^p(\R^n)}^{p} \right)^{1/p},
\end{align*}
c.f. for instance \cite{Triebel} or \cite{BCD}.
For later use, we recall that $B_{1,1}^0(\R^n) \subset L^1(\R^n)$ (c.f. Theorem 2.41 in \cite{BCD}).
For $s> 0, \ s\notin \N$ this space coincides with the fractional Sobolev space $W^{s,p}(\R^n)$, which is defined by the norm
\begin{align*}
\|u\|_{W^{s,p}(\R^n)}^p = \|u\|_{L^p(\R^n)}^p + \|\nabla^k u\|_{L^p(\R^n)}^p + \int\limits_{\R^n}\int\limits_{\R^n} \frac{|u(x)-u(y)|^p}{|x-y|^{n+\sigma p}} dy dx,
\end{align*}
where $s=k+\sigma$, $k\in \N$, $\sigma\in (0,1)$ (c.f. \cite{BM12} for a comparison of several fractional Sobolev type function spaces).

\subsection{Proof of Theorem \ref{thm:lower_bd}}
\label{sec:thm1}

We begin the discussion of the relation between scaling and regularity by presenting the proof of Theorem \ref{thm:lower_bd}.

\begin{proof}[Proof of Theorem \ref{thm:lower_bd}]
We argue by regularization (through a frequency cut-off in the corresponding Littlewood-Paley decomposition) and a spacial cut-off (in order to ensure the validity of the boundary conditions).

Let $u$ be a solution to \eqref{eq:conv_int}. As a preliminary step, to simplify the set-up, we shift the problem \eqref{eq:conv_int} by introducing $v(x)=u(x)-Mx-b$ for $b\in \R^n$, which solves
\begin{align}
\label{eq:conv_inta}
\begin{split}
\nabla v &\in \tilde{K} \mbox{ a.e. in } \Omega,\\
\nabla v &= 0 \mbox{ a.e. in } \R^n \setminus \overline{\Omega},
\end{split}
\end{align}
where $\tilde{K}=K-M:=\{ N-M: \ N \in K\}$. The constant $b\in \R^n$ is chosen such that $v \in H^{s+1}(\R^n)$ and $\supp(v) \subset \overline{\Omega}$. We seek to show that $s\leq \mu$.
Rewritten in terms of the shifted functions $v$, the energy \eqref{eq:Eeps} turns into
\begin{align}
\label{eq:E_new}
\tilde{E}_{\epsilon}= \min\limits_{\nabla v = 0 \text{ a.e. in } \R^n \setminus \overline{\Omega}} \int\limits_{\Omega} \dist^2(\nabla v, \tilde{K}) dx + \epsilon^2 \int\limits_{\Omega} |\nabla^2 v|^2 dx.
\end{align} 
For $P_{<N}, P_{\geq N}$ denoting the standard Littlewood-Paley projector (c.f. the definition in \eqref{eq:LP}), in the sequel, we use the following notation:
\begin{align*}
(\nabla u)_{<N}:= P_{<N}(\nabla u), \
(\nabla u)_{\geq N}:= P_{\geq N}(\nabla u). 
\end{align*}
Given the function $v$, we cut-off its high frequencies and correct the frequency cut-off of $v$ so that it satisfies the desired boundary conditions, i.e. we consider the function
\begin{align*}
\tilde{v}_{N}(x):= \eta_{\delta}(x) v_{<N}(x) ,
\end{align*}
where $\eta_{\delta}$ is a smooth, positive cut-off function which is equal to one in $\Omega_{2\delta}:=\{x\in \Omega: \dist(x,\partial \Omega)>2\delta\}$, which vanishes outside of $\Omega_{\delta}:=\{x\in \Omega: \dist(x,\partial \Omega)>\delta\}$ and which satisfies
\begin{align}
\label{eq:cut-off}
|\nabla \eta_{\delta}| \leq \frac{C}{\delta}, \ |\nabla^2 \eta_{\delta}| \leq \frac{C}{\delta^2}.
\end{align}
The value of the parameter $\delta$ will be determined in the sequel. Moreover, here and in the following arguments, the constant $C>0$ will be generic and might change from line to line, but will always be independent of $\delta, N, \epsilon$.
We observe that for $v $ with $\nabla v \in L^{\infty}(\R^n)\cap H^s(\R^n)$ Bernstein estimates (c.f. the Appendix in \cite{Tao06}, and Section 2.1.1. in \cite{BCD}) yield
\begin{align}
\label{eq:Bern_a}
\begin{split}
\|\nabla^2 v_{<N}\|_{L^2(\Omega)}
&\leq  \|\nabla^2 v_{<N}\|_{L^2(\R^n)}
\leq  N^{1-s} \|\nabla v\|_{H^{s}(\R^n)},\\
\|(\nabla v)_{\geq N}\|_{L^2(\Omega)} 
&\leq  \|(\nabla v)_{\geq N}\|_{L^2(\R^n)} 
\leq N^{-s} \|\nabla v\|_{H^{s}(\R^n)}.
\end{split}
\end{align}

We now argue in two steps and first estimate the size of $\|\dist(\nabla \tilde{v}_{N}, \tilde{K})\|_{L^2(\Omega)}^2$. We show that this quantity is controlled by $C\epsilon^{2s}$ for some uniform constant $C>1$. Afterwards we prove a similar bound for $\epsilon^2\|\nabla^2 \tilde{v}_{N}\|_{L^2(\Omega)}^2$.  From this we will derive a bound on the possible regularity of $\nabla v$ (and hence for $\nabla u$).\\

\emph{Step 1: Estimate for $\|\dist(\nabla \tilde{v}_{N},\tilde{K})\|_{L^2(\Omega)}^2$.}
Using the triangle inequality, we estimate:
\begin{align}
\label{eq:main}
\begin{split}
\|\dist( \nabla \tilde{v}_{N}, \tilde{K}) \|_{L^2(\Omega)}^2
&\leq C\left(\| \dist(\eta_{\delta} \nabla v_{<N}, \tilde{K}) \|_{L^2(\Omega)}^2 + \|v_{<N} \nabla \eta_{\delta}\|_{L^2(\Omega)}^2\right) \\
&\leq C\left(\|\dist(\nabla v, \tilde{K}) \|_{L^2(\Omega)}^2 
+ \|\dist((1-\eta_{\delta})\nabla v, \tilde{K})\|_{L^2(\Omega)}^2 \right.\\
& \quad \left. + \|\nabla v - \nabla v_{<N}\|_{L^2(\Omega)}^2 
+ \|v_{<N} \nabla \eta_{\delta}\|_{L^2(\Omega)}^2 
\right)\\
&\leq C\left(\delta+ \|v_{<N} \nabla \eta_{\delta}\|_{L^2(\R^n)}^2 
+ \|\nabla v - \nabla v_{<N}\|_{L^2(\R^n)}^2
\right).
\end{split}
\end{align}
Here we used that $v$ is a solution to \eqref{eq:conv_inta} and that since $\nabla v \in L^{\infty}(\Omega)$ and since $\tilde{K}$ is bounded, the support condition for $\eta_{\delta}$ implies
\begin{align*}
\int\limits_{\Omega}\dist^2((1-\eta_{\delta})\nabla v, \tilde{K}) dx \leq C \delta.
\end{align*}
We estimate the two terms on the right hand side of \eqref{eq:main} separately: First, we apply the Bernstein estimates from \eqref{eq:Bern_a} to infer
\begin{align}
\label{eq:aux_a}
 \|\nabla v - \nabla v_{<N}\|_{L^2(\R^n)}^2
 = \|\nabla v_{\geq N}\|_{L^2(\R^n)}^2
 \leq C N^{-2s}\|\nabla v\|_{H^{s}(\R^n)}^2.
\end{align}
Next, we estimate
\begin{align}
\label{eq:aux_b}
\|v_{<N} \nabla \eta_{\delta}\|_{L^2(\R^n)}^{2}
\leq C (\|v \nabla \eta_{\delta}\|_{L^2(\R^n)}^2 + \|v_{\geq N}\nabla \eta_{\delta}\|_{L^2(\R^n)}^2).
\end{align}
In order to bound the second term on the right hand side of \eqref{eq:aux_b}, we note that by the bounds for $\eta_{\delta}$ (c.f. \eqref{eq:cut-off}) and Bernstein estimates
\begin{align}
\label{eq:aux_b1}
\begin{split}
\|v_{\geq N} \nabla \eta_{\delta}\|_{L^2(\R^n)}^2
&\leq C\delta^{-2} \|v_{\geq N}\|_{L^2(\R^n)}^2
\leq C \delta^{-2} N^{-2-2s} \|\nabla v\|_{H^{s}(\R^n)}^2.
\end{split}
\end{align}
For the first term on the right hand side of \eqref{eq:aux_b}, using that $\supp(\nabla \eta_{\delta}) \subset \Omega_{\delta} \setminus \Omega_{2\delta}$, we obtain 
\begin{align}
\label{eq:aux_b2}
\|v \nabla \eta_{\delta}\|_{L^2(\R^n)}^2 
\leq C \delta^{-2} \|v\|_{L^2(\Omega_{\delta}\setminus \Omega_{2\delta})}^2
\leq C  \|\nabla v\|_{L^2(\Omega \setminus \Omega_{2\delta})}^2
\leq C \delta^{2s} \|\nabla v\|_{H^{s}(\R^n)}^2.
\end{align}
Indeed, the last estimate in this chain of inequalities can be inferred by using the support assumption for $\nabla v$ and a fractional Poincar\'e inequality (c.f. Lemma \ref{lem:frac_poinc_L2} in Section \ref{sec:append}).

Combining the bounds from \eqref{eq:aux_a}-\eqref{eq:aux_b2}, we thus obtain as the estimate for the elastic energy
\begin{align*}
\|\dist(\nabla \tilde{v}_{N}, \tilde{K})\|_{L^2(\Omega)}^2
\leq C (N^{-2s} + \delta^{-2}N^{-2-2s} + \delta^{2s})\|\nabla v\|_{H^{s}(\Omega)}^2 + C \delta.
\end{align*}
Choosing $N = \delta^{-1}= \epsilon^{-1}$, hence implies the bound
\begin{align}
\label{eq:est_wells}
\|\dist(\nabla \tilde{v}_{N}, \tilde{K})\|_{L^2(\Omega)}^2
\leq C (\epsilon^{2s} + \epsilon).
\end{align}

\emph{Step 2: Conclusion.}
Next we estimate the surface energy:
\begin{align}
\label{eq:surf_a}
\begin{split}
\|\nabla^2 \tilde{v}_{N}\|_{L^2(\Omega)}^2
&\leq C (\|\nabla^2 v_{<N}\|_{L^2(\R^n)}^2 + \|v_{<N}\nabla^2 \eta_{\delta}\|_{L^2(\R^n)}^2 + \|\nabla \eta_{\delta} \nabla v_{<N}\|_{L^2(\R^n)}^2)\\
&\leq C(N^{2-2s}\|\nabla v\|_{H^{s}(\R^n)}^2 + \|v_{<N}\nabla^2 \eta_{\delta}\|_{L^2(\R^n)}^2 + \|\nabla \eta_{\delta} \nabla v_{<N}\|_{L^2(\R^n)}^2).
\end{split}
\end{align}
We again estimate the two last terms separately: Using a fractional Poincar\'e inequality once more in combination with Bernstein estimates, we obtain for the third term in \eqref{eq:surf_a}
\begin{align}
\label{eq:surf_b}
\begin{split}
\|\nabla \eta_{\delta} \nabla v_{<N}\|_{L^2(\R^n)}^2
&\leq C (\|\nabla \eta_{\delta} \nabla v\|_{L^2(\Omega)}^2 + \|\nabla \eta_{\delta} \nabla v_{\geq N}\|_{L^2(\Omega)}^2)\\
&\leq C \delta^{-2} \delta^{2s} \|\nabla v\|_{H^{s}(\R^n)}^2 + C \delta^{-2}N^{-2s} \|\nabla v\|_{H^{s}(\R^n)}^2.
\end{split}
\end{align}
For the second term in \eqref{eq:surf_a}, we similarly note by Poincar\'e and Bernstein
\begin{align}
\label{eq:surf_c}
\begin{split}
\|v_{<N} \nabla^2 \eta_{\delta}\|_{L^2(\Omega)}^2
&\leq C \|v\nabla^2 \eta_{\delta} \|_{L^2(\R^n)}^2 
+ C \|v_{\geq N} \nabla^2 \eta_{\delta}\|_{L^2(\R^n)}^2\\
& \leq C \|\nabla v\|_{H^{s}(\R^n)}^2 (\delta^{-2+2s} + \delta^{-4}N^{-2-2s}).
\end{split}
\end{align}
Combining \eqref{eq:surf_a}-\eqref{eq:surf_c} we therefore obtain
\begin{align}
\label{eq:surf_d}
\epsilon^2\|\nabla^2 \tilde{v}_N\|_{L^2(\Omega)}^2
\leq C\epsilon^2(N^{2-2s}+\delta^{-2+2s}+\delta^{-4}N^{-2-2s} + \delta^{-2}N^{-2s})\|\nabla v\|_{H^{s}(\R^n)}^2.
\end{align}
Finally, combining \eqref{eq:surf_d} with \eqref{eq:est_wells} and choosing $N^{-1}=\delta = \epsilon$, we conclude
\begin{align}
\label{eq:concl_a}
\epsilon^{2\mu} \leq C \|\dist(\nabla \tilde{v}_{N}, \tilde{K})\|^2_{L^2(\Omega)} + C \epsilon^2 \|\nabla^2 \tilde{v}_N\|_{L^2(\Omega)}^2 \leq C \epsilon^{2s} \|\nabla v\|_{H^{s}(\R^n)}^2 + C \epsilon.
\end{align}
As $\|\nabla v\|_{H^{s}(\R^n)}^2$ is bounded independently of $\epsilon$, we observe that, for $\epsilon \in (0,1)$ sufficiently small and $\mu \in (0,\frac{1}{2})$, \eqref{eq:concl_a} can only hold as long as $s \leq \mu$. This concludes the argument.
\end{proof}

\begin{rmk}
\label{rmk:trace}
The restriction $\mu \in (0,1/2)$ enters rather naturally and is connected to the presence/absence of a trace. For most (though not for all, c.f. Section 7 in \cite{RZZ16}) convex integration solutions, one would expect that no trace estimates are available, the prescribed boundary data enforcing too strong oscillations in $\nabla u$ close to the boundary $\partial \Omega$. For the convex integration solutions in \cite{DMP08a} the presence of a trace for instance provides the sharp obstructions for the existence of convex integration solutions.
\end{rmk}

\subsection{Variations on Theorem \ref{thm:lower_bd}}
\label{sec:var}

Instead of working in $L^2$ based function spaces, we can also work with a Littlewood-Paley decomposition and the associated Bernstein estimates in other function spaces, i.e. for different energy functionals.\\

Consider for instance the (sharp interface) energy functional
\begin{align}
\label{eq:Eeps1}
E_{\epsilon,1}
= \min\limits_{\nabla u = M \text{ a.e. in } \R^n \setminus \overline{\Omega}} \left\{ \int\limits_{\Omega}\dist^2(\nabla u, K) dx 
+ \epsilon |\nabla^2 u|(\Omega)\right\},
\end{align}
where $|\nabla^2 u|(\Omega)$ denotes the variation measure of the distributional derivative of $\nabla u$.

As in Theorem \ref{thm:lower_bd} we obtain the following bound for possible convex integration solutions:

\begin{prop}
\label{prop:scaling_BV}
Let $E_{\epsilon,1}$ be the energy from \eqref{eq:Eeps1}. Assume that there exist constants $C>1$ and $\mu \in (0,\frac{1}{2})$ such that for all $\epsilon \in (0,\epsilon_0)$
\begin{align*}
  \epsilon^{2\mu} \leq C E_{\epsilon,1}.
\end{align*} 
Suppose that $u$ is a solution to \eqref{eq:conv_int}.
If for some $b \in \R^n$ and $s\in \R$ we have
$v(x)=u(x)-Mx-b \in B_{1,1}^{1+s}(\R^n)$, $\supp(v)\subset \overline{\Omega}$ and $\nabla v \in L^{\infty}(\R^n)$, then $s\leq 2\mu$.
\end{prop}

Here $B^{1+s}_{1,1}(\R^n)$ denotes the standard Besov space (defined for instance through a suitable Littlewood-Paley decomposition, c.f. Section \ref{sec:FS} and \cite{BCD}). Loosely speaking the proposition states, that on a Besov scale, we can at most have $\nabla u \in (B^{2\mu}_{1,1})_{loc}(\R^n)$.

\begin{proof}
As a first preliminary step, we again switch to the situation where
\begin{align*}
v(x) = u(x)-Mx-b,
\end{align*}
and $b\in \R^n$ is chosen such that $v\in B_{1,1}^{1+s}(\R^n)$ and $\supp(v)\subset \overline{\Omega}$.
The modified energy wells turn into $\tilde{K}:= K-M$, and the energy \eqref{eq:Eeps1} becomes 
\begin{align}
\label{eq:tEps1}
\tilde{E}_{\epsilon,1}:= \min\limits_{\nabla v = 0 \text{ a.e. in } \R^n \setminus \overline{\Omega}} \left\{ \|\dist(\nabla v, \tilde{K})\|_{L^2(\Omega)}^2 + \epsilon |\nabla^2 v|(\Omega) \right\}.
\end{align}

Given a convex integration solution $u$ of \eqref{eq:conv_int} or $v$ of \eqref{eq:conv_inta}, we again consider the functions
\begin{align*}
\tilde{v}_{N}(x):= \eta_{\delta}(x) v_{<N}(x) ,
\end{align*}
where $\eta_{\delta}$ is a smooth, positive cut-off function which is equal to one in $\Omega_{2\delta}:=\{x\in \Omega: \dist(x,\partial \Omega)>2\delta\}$, which vanishes outside of $\Omega_{\delta}:=\{x\in \Omega: \dist(x,\partial \Omega)>\delta\}$ and which satisfies
\begin{align}
\label{eq:cut-off1}
|\nabla \eta_{\delta}| \leq \frac{C}{\delta}, \ |\nabla^2 \eta_{\delta}| \leq \frac{C}{\delta^2}.
\end{align}
The value of the parameter $\delta$ will be determined in the sequel. 
Using Bernstein type inequalities \cite{BCD} and the embeddings $B_{1,1}^{0}(\R^n)\subset L^{1}(\R^n)$ (c.f. Theorem 2.41 in \cite{BCD}) and $W^{1,1}(\Omega) \subset BV(\Omega)$, 
for $v $ with $\nabla v \in L^{\infty}(\R^n)\cap B^s_{1,1}(\R^n)$ we deduce
\begin{align}
\label{eq:Bern}
\begin{split}
|\nabla (\nabla v)_{<N}|(\Omega)
&\leq  \|(\nabla v)_{<N}\|_{\dot{W}^{1,1}(\R^n)} 
\leq  C\|(\nabla^2 v)_{<N}\|_{B^{0}_{1,1}(\R^n)} 
\leq  CN^{1-s} \|\nabla v\|_{B^{s}_{1,1}(\R^n)},\\
\|(\nabla v)_{\geq N}\|_{L^1(\Omega)} 
&\leq  \|(\nabla v)_{\geq N}\|_{L^1(\R^n)}
\leq C\|(\nabla v)_{\geq N}\|_{B^{0}_{1,1}(\R^n)} 
\leq CN^{-s} \|\nabla v\|_{B^{s}_{1,1}(\R^n)}.
\end{split}
\end{align}
With this at hand, using that
\begin{align*}
\nabla \tilde{v}_{N} = \eta_{\delta} \nabla v_{<N} + v_{<N} \nabla \eta_{\delta},
\end{align*}
and
using the properties of the cut-off $\eta_{\delta}$ (c.f. \eqref{eq:cut-off1}), we estimate as follows:
\begin{align}
\label{eq:L1_grad1}
\begin{split}
&\|\dist(\nabla \tilde{v}_{N}, \tilde{K})\|_{L^2(\Omega)}^2
\leq 2\|\dist(\eta_{\delta} \nabla v_{<N}, \tilde{K})\|_{L^2(\Omega)}^2 
+ 2\|v_{<N} \nabla \eta_{\delta}\|_{L^2(\Omega)}^2\\
&\leq 2\|\dist(\eta_{\delta}  \nabla v_{<N}, \tilde{K})\|_{L^2(\Omega)}^2 
+ 2\|v_{<N} \nabla \eta_{\delta}\|_{L^{\infty}(\R^n)}\|v_{<N} \nabla \eta_{\delta}\|_{L^{1}(\R^n)}\\
&\leq 2\|\dist(\eta_{\delta}  \nabla v_{<N}, \tilde{K})\|_{L^2(\Omega)}^2 
+ 2C_{\nu} \left( \frac{1}{\delta N^{2-\nu}}+\delta\right)\|\nabla v\|_{L^{\infty}(\R^n)}^{2},
\end{split}
\end{align}
where $\nu >0$ is an arbitrarily small constant.
In order to infer this, we used the bounds
\begin{align}
\label{eq:aux_1}
\|v_{<N} \nabla \eta_{\delta}\|_{L^{q}(\R^n)} 
& \leq C( \|v \nabla \eta_{\delta}\|_{L^{q}(\R^n)} 
+  \|v_{\geq N} \nabla \eta_{\delta}\|_{L^{q}(\R^n)}
)
\end{align}
for $q\in \{1,\infty\}$. By Bernstein estimates, we control the second term in \eqref{eq:aux_1} for $q=\infty$
\begin{align}
\label{eq:aux_1a}
\begin{split}
 \|v_{\geq N} \nabla \eta_{\delta}\|_{L^{\infty}(\R^n)}&\leq C \delta^{-1} N^{\frac{n}{p}}\|v_{\geq N}\|_{L^p(\R^n)}
\leq C \delta^{-1} N^{\frac{n}{p}} N^{-1} \||\nabla| v_{\geq N}\|_{L^p(\R^n)}\\
&\leq C \delta^{-1} N^{\frac{n}{p}} N^{-1} \|\nabla v_{\geq N}\|_{L^p(\R^n)}
\leq C \delta^{-1} N^{\frac{n}{p}} N^{-1} \|\nabla v\|_{L^p(\R^n)}\\
& \leq C \delta^{-1} N^{-1+\nu/2} \|\nabla v\|_{L^{\infty}(\R^n)},
\end{split}
\end{align}
with $\frac{\nu}{2}:= \frac{n}{p}$ (and $p>1$ can be chosen arbitrarily large),
and for $q=1$
\begin{align}
\label{eq:aux_2}
\begin{split}
&\|v_{\geq N}\nabla \eta_{\delta}\|_{L^1(\R^n)} \leq \|\nabla \eta_{\delta}\|_{L^{\infty}(\R^n)} \|v_{\geq N}\|_{L^{\infty}(\R^n)}|\Omega_{\delta} \setminus \Omega_{2\delta}|\\
& \qquad \leq C \frac{1}{\delta} \frac{1}{N^{1-\nu/2}}\|\nabla v\|_{L^{\infty}(\R^n)}\delta = \frac{C}{N^{1-\nu/2}}\|\nabla v\|_{L^{\infty}(\R^n)}.
\end{split}
\end{align}
The respective first terms on the right hand side of \eqref{eq:aux_1} are controlled by invoking the fundamental theorem
\begin{align*}
\|v \nabla \eta_{\delta}\|_{L^{\infty}(\R^n)} 
&\leq C \delta^{-1}\|v \|_{L^{\infty}(\Omega_{\delta} \setminus \Omega_{2\delta})} \leq C \delta^{-1} \delta  \|\nabla v\|_{L^{\infty}(\Omega\setminus \Omega_{2\delta})} ,\\
\|v \nabla \eta_{\delta}\|_{L^{1}(\R^n)} 
&\leq C \delta^{-1}\|v \|_{L^{1}(\Omega_{\delta} \setminus \Omega_{2\delta})} 
\leq C \delta^{-1} \delta  \|\nabla v\|_{L^{1}(\Omega\setminus \Omega_{2\delta})}
\leq C  \delta  \|\nabla v\|_{L^{\infty}(\Omega\setminus \Omega_{2\delta})} .
\end{align*}
Further, we observe that
\begin{align}
\label{eq:well_en_a}
\begin{split}
&\|\dist(\eta_{\delta} \nabla v_{<N}, \tilde{K})\|_{L^2(\Omega)}^2
\leq C\|\dist(\eta_{\delta} \nabla v_{<N}, \tilde{K})\|_{L^1(\Omega)} (\| \nabla v_{<N}\|_{L^{\infty}(\Omega)}+C_{\tilde{K}})\\
&\leq C\left(\|\dist(\eta_{\delta} \nabla v, \tilde{K})\|_{L^1(\Omega)} 
+ \|\nabla v - \nabla v_{<N}\|_{L^1(\Omega)} \right)
(\| \nabla v_{<N}\|_{L^{\infty}(\R^n)}+C_{\tilde{K}})\\
&\leq C\left(\|\dist( \nabla v, \tilde{K})\|_{L^1(\Omega)} + \|\dist((1-\eta_{\delta}) \nabla v, \tilde{K})\|_{L^1(\Omega)} 
+ \|\nabla v - \nabla v_{<N}\|_{L^1(\R^n)} \right)\\
&\qquad \times (\| \nabla v_{<N}\|_{L^{\infty}(\R^n)}+C_{\tilde{K}})\\
&\leq C\left(\delta
+ N^{-s}\|\nabla v_{\geq N}\|_{B^{s}_{1,1}(\R^n)} \right)(N^{\frac{\nu}{2}}
\| \nabla v\|_{L^{\infty}(\R^n)}+C_{\tilde{K}})\\
&\leq C\left(\delta
+ N^{-s}\|\nabla v\|_{B^{s}_{1,1}(\R^n)} \right)(N^{\frac{\nu}{2}}
\| \nabla v\|_{L^{\infty}(\R^n)}+C_{\tilde{K}}).
\end{split}
\end{align}
Combining \eqref{eq:L1_grad1}-\eqref{eq:well_en_a}, we infer for the elastic energy
\begin{align}
\label{eq:elast_a1}
\begin{split}
\|\dist(\nabla \tilde{v}_N, \tilde{K}) \|_{L^2(\Omega)}^2 
&\leq  C\left(\delta N^{\frac{\nu}{2}}
+ N^{-s+\frac{\nu}{2}}\|\nabla v\|_{B^{s}_{1,1}(\R^n)}+C\delta^{-1}N^{-2+\nu} \right)\\
& \quad \times 
(\| \nabla v\|_{L^{\infty}(\R^n)}+C_{\tilde{K}}) .
\end{split}
\end{align}
For the surface energy we estimate 
\begin{align}
\label{eq:L1_grad_2}
\begin{split}
|\nabla^2 \tilde{v}_{N}|(\Omega)
&\leq \|\nabla^2 v_{<N}\|_{L^1(\Omega)} + 2 \|\nabla \eta_{\delta}  \nabla v_{<N} \|_{L^1(\Omega)} + \| v_{<N} \nabla^2 \eta_{\delta}\|_{L^1(\Omega)}\\
&\leq \|\nabla^2 v_{<N}\|_{L^1(\Omega)} + 2 \|\nabla \eta_{\delta}  \nabla v_{\geq N} \|_{L^1(\R^n)}
+ 2 \|\nabla \eta_{\delta}  \nabla v \|_{L^1(\R^n)}\\
& \quad  + \|v_{\geq N} \nabla^2 \eta_{\delta} \|_{L^1(\R^n)}
+ 2 \|v \nabla^2 \eta_{\delta}  \|_{L^1(\R^n)} 
 \\
&\leq C(\|\nabla^2 v_{<N}\|_{L^1(\Omega)} + \frac{1}{\delta}\|\nabla v_{\geq N}\|_{L^1(\R^n)} + \frac{1}{\delta^2}\| v_{\geq N}\|_{L^1(\R^n)}\\
& \quad + \frac{1}{\delta}\|\nabla v\|_{L^1(\Omega\setminus \Omega_{2\delta})} + \frac{1}{\delta^2}\| v\|_{L^1(\Omega\setminus \Omega_{2\delta})})
\\
&\leq C( N^{1-s}\|\nabla v\|_{B^{s}_{1,1}(\R^n)} + \frac{2}{\delta N^{s}}\|\nabla v\|_{B^{s}_{1,1}(\R^n)} + \frac{1}{\delta^2 N^{1+s}}\| \nabla v \|_{B^{s}_{1,1}(\R^n)}\\
& \quad + \frac{2 \delta^{s}}{\delta}\|\nabla v\|_{B^s_{1,1}(\R^n)} + \frac{\delta^{1+s}}{\delta^2}\| \nabla v\|_{B^{s}_{1,1}(\R^n)}).
\end{split}
\end{align}
Here we have used some of the Bernstein type estimates from \eqref{eq:Bern} in combination with a fractional Poincar\'e inequality (c.f. Section \ref{sec:append}).
Combining \eqref{eq:elast_a1}, \eqref{eq:L1_grad_2} and the assumed lower bound for $E_{\epsilon,1}$ and choosing $N=\delta^{-1}$ yields
\begin{align}
\label{eq:lower_bd}
\begin{split}
\epsilon^{2\mu} &\leq \tilde{E}_{\epsilon,1} 
\leq \|\dist(\nabla\tilde{v}_{N},K)\|_{L^2(\Omega)}^2 + \epsilon |\nabla^2 \tilde{v}_{N}|(\Omega)\\
&\leq C(\nu,\|\nabla v\|_{L^{\infty}(\R^n)}, \|\nabla v\|_{B^{s}_{1,1}(\R^n)}) \left(N^{-1}+ N^{-s+\frac{\nu}{2}} + N^{-1+\nu} + \epsilon N^{1-s}\right).
\end{split}
\end{align}

Choosing $N \sim \epsilon^{-1}$, we obtain
\begin{align*}
\epsilon^{2\mu} &\leq \tilde{E}_{\epsilon,1} 
\leq  C(\nu,\|\nabla v\|_{L^{\infty}(\R^n)}, \|\nabla v\|_{B^{s}_{1,1}(\R^n)})( \epsilon^{s - \frac{\nu}{2}} + \epsilon^{1-\nu}).
\end{align*}
As a consequence, by choosing $\nu$ depending on $\mu$ sufficiently small and considering the limit $\epsilon \rightarrow 0$, we infer that on a Besov scale no ($L^{\infty}$-) convex integration solution can be more than $B^{2\mu}_{1,1}(\R^n)$-regular for $\mu\in (0,1/2)$.
\end{proof}

More generally, the previous situation can be studied with regularisations in the Sobolev spaces $W^{s,p}$. 
In this setting we have the following result:

\begin{prop}
\label{prop:scaling_p}
Let $1< p \leq 2$ and denote by $E_{\epsilon,p}$ the energy 
\begin{align}
\label{eq:Eepsp}
E_{\epsilon,p}
= \min\limits_{\nabla u = M \text{ a.e. in } \R^n \setminus \overline{\Omega} } \left\{ \int\limits_{\Omega}\dist^2(\nabla u, K) dx 
+ \epsilon^p \|\nabla^2 u\|_{L^p(\Omega)}^p \right\}.
\end{align}
Assume that there exists constants $C_p>1$ and $\mu \in (0,1/2)$ such that for all $\epsilon \in (0,\epsilon_0)$
\begin{align*}
  \epsilon^{2\mu} \leq C_p E_{\epsilon,p}.
\end{align*} 
Suppose that $u$ is a solution to \eqref{eq:conv_int}. If for some $b\in \R^n$ and $s\in \R$ it holds that $v(x)=u(x)-Mx-b \in W^{s,p}(\R^n)$ with $\supp(v)\subset \overline{\Omega}$ and $\nabla v \in L^{\infty}(\R^n)$, then $s\leq \frac{2\mu}{p}$.
\end{prop}

In other words, the scaling behaviour yields an upper bound on the regularity of convex integration solutions, i.e. on a Sobolev $W^{s,p}_{loc}$ scale, we can at most have $\nabla u \in W^{\frac{2\mu}{p},p}_{loc}(\R^n)\cap L^{\infty}(\R^n)$.

The proof proceeds as in the $BV$ setting by using Bernstein type estimates.

\begin{proof}
As in the previous proofs, we reduce to the shifted problem by setting $v(x)=u(x)-Mx-b$ for which we assume that for some $b\in \R^n$ we have $v\in W^{s,p}(\R^n)$ and $\supp(v) \subset \overline{\Omega}$. Hence, if $u$ is a solution to \eqref{eq:conv_int}, then $v$ is a solution to \eqref{eq:conv_inta}. The correspondingly modified energy reads
\begin{align*}
\tilde{E}_{\epsilon,p}= \min\limits_{\nabla v = 0 \text{ a.e. in } \R^n \setminus \overline{\Omega}} \left\{ \int\limits_{\Omega} \dist^2(\nabla v, \tilde{K})dx + \epsilon^p \|\nabla^2 v\|_{L^p(\Omega)}^p \right\}.
\end{align*}
As in the proof of Proposition \ref{prop:scaling_BV}, we use the standard Littlewood-Paley decomposition to define
\begin{align*}
\tilde{v}_N(x):= \eta_{\delta}(x) v_{<N}(x),
\end{align*}
and estimate for $N=\delta^{-1}$
\begin{align*}
&\|\dist(\nabla \tilde{v}_{N}, \tilde{K})\|_{L^2(\Omega)}^2
\leq C(\|\dist(\nabla v_{<N},\tilde{K})\|_{L^2(\Omega)}^2 + \|v_{<N} \nabla \eta_{\delta} \|_{L^2(\Omega)}^2 + \delta)\\
&\leq C(\|\dist(\nabla v_{<N},\tilde{K})\|_{L^2(\Omega)}^2 + \|v_{<N} \nabla \eta_{\delta} \|_{L^1(\R^n)}\|v_{<N} \nabla \eta_{\delta} \|_{L^{\infty}(\R^n)} + \delta)\\
&\leq C_{\nu}(\|\dist(\nabla v_{<N},\tilde{K})\|_{L^2(\Omega)}^2 + \delta^{1-\nu}\|\nabla v\|_{L^{\infty}(\Omega)}^2 + \delta),
\end{align*}
where we used similar estimates as in \eqref{eq:aux_1}, \eqref{eq:aux_2} in the proof of Proposition \ref{prop:scaling_BV}.
Furthermore,
\begin{align}
\label{eq:dist_estp}
\begin{split}
\|\dist(\nabla v_{<N},\tilde{K})\|_{L^2(\Omega)}^2
&\leq \|\dist(\nabla v_{<N}, \tilde{K})\|_{L^p(\Omega)}^p \|\dist(\nabla v_{<N}, \tilde{K})\|_{L^{\infty}(\Omega)}^{2-p}\\
&\leq C(\|\dist(\nabla v, \tilde{K})\|_{L^p(\Omega)} + \|\nabla v- \nabla v_{<N}\|_{L^p(\R^n)})^p \\
& \quad \times (N^{\nu}\|\nabla u\|_{L^{\infty}(\R^n)}^{2-p}+C_{\tilde{K}})  \\
&\leq  C \|\nabla v_{\geq N}\|_{L^p(\R^n)}^p (N^{\nu}\|\nabla u\|_{L^{\infty}(\R^n)}^{2-p}+C_{\tilde{K}})\\
&\leq C N^{-ps}\|\nabla v\|_{W^{s,p}(\R^n)}^p (N^{\nu}\|\nabla u\|_{L^{\infty}(\R^n)}^{2-p}+C_{\tilde{K}}).
\end{split}
\end{align}
Here we used that by Bernstein estimates, Calder\'on-Zygmund estimates, and the compact support of $\nabla v$,
\begin{align*}
\|\nabla v_{<N}\|_{L^{\infty}(\R^n)}^{2-p}\leq C_{\nu} N^{\nu} \|\nabla v_{<N}\|_{L^p(\R^n)}^{2-p} \leq C_{\nu} N^{\nu} \|\nabla v\|_{L^p(\R^n)}^{2-p} \leq C_{\nu} N^{\nu} \|\nabla v\|_{L^{\infty}(\R^n)}^{2-p}
\end{align*}
for $\nu = \nu(p)>0$ being arbitrarily small.
Thus, for $N^{-1}=\delta$ the elastic energy is controlled by
\begin{align*}
\|\dist(\nabla \tilde{v}_{N}, \tilde{K})\|_{L^2(\Omega)}^2
\leq C( \delta^{ps - \nu}\|\nabla v\|_{W^{s,p}(\R^n)}^p +  \delta^{1-\nu}\|\nabla v\|_{L^{\infty}(\R^n)}^2 + \delta)
\end{align*}
with $C=C(\tilde{K}, \nu, n, \|\nabla v\|_{L^{\infty}(\R^n)})$.
For the surface energy we estimate (by using Poincar\'e and Bernstein)
\begin{align}
\label{eq:surface:p}
\begin{split}
\|\nabla^2 \tilde{v}_N\|_{L^p(\Omega)}^p
&\leq C (\|\nabla^2 v_{<N}\|_{L^p(\Omega)}^p + \|\nabla \eta_{\delta}\nabla v_{<N}\|_{L^p(\Omega)}^p + \|v_{<N}\nabla^2 \eta_{\delta}\|_{L^p(\Omega)}^p)\\
&\leq C (\|\nabla^2 v_{< N}\|_{L^p(\R^n)}^p + \|\nabla \eta_{\delta}\nabla v_{\geq N}\|_{L^p(\R^n)}^p + \|v_{\geq N}\nabla^2 \eta_{\delta}\|_{L^p(\R^n)}^p\\
& \quad +  \|\nabla \eta_{\delta}\nabla v\|_{L^p(\R^n)}^p + \|v\nabla^2 \eta_{\delta}\|_{L^p(\R^n)}^p)\\
& \leq C (N^{p-sp}+\delta^{-p}N^{-sp} + \delta^{-2p}N^{-sp-p} \\
& \qquad + \delta^{-p}\delta^{sp}+ \delta^{-2p}\delta^{p+sp})\|\nabla v\|_{W^{s,p}(\R^n)}^p\\
& \leq C \delta^{-p+sp}\|\nabla v\|_{W^{s,p}(\R^n)}^p,
\end{split}
\end{align}
where we have again chosen $\delta=N^{-1}$.
Thus, by recalling the assumed lower bound for the energy, we infer
\begin{align*}
\epsilon^{2\mu} \leq \tilde{E}_{\epsilon,p} \leq C(\nu, \|\nabla v\|_{L^{\infty}(\Omega)}, \|\nabla v\|_{W^{s,p}(\Omega)})( \epsilon^{p} N^{p-sp}+ C N^{-1+\nu} + C N^{-1} + N^{-ps+\nu}),
\end{align*}
where $C\geq 1$.
Further setting $N=\epsilon^{-1}$, we hence deduce that
\begin{align*}
\epsilon^{2\mu} &\leq \tilde{E}_{\epsilon,p} 
\leq C (\epsilon^{sp - \nu}  + \epsilon^{1-\nu}).
\end{align*}
As a consequence, for $\mu \in (0,1/2)$ on a $W^{s,p}$ Sobolev scale no convex integration solution can be more than $W^{\frac{2\mu}{p},p}$-regular (which can be observed by choosing $\nu$ depending on $\mu$ sufficiently small and by passing to the limit $\epsilon \rightarrow 0$).
\end{proof}

\begin{rmk}
\label{rmk:break_scaling1}
Similarly as in \cite{RZZ16} and in \cite{Si}, we have thus obtained that if a family of energy functionals is controlled, then the possible $W^{s,p}$ Sobolev regularity is determined in terms of the product $sp$ and not by $s,p$ individually. Here the $L^{\infty}$ bound was crucial, as it allowed us to ``break scaling".
\end{rmk}

\begin{rmk}
\label{rmk:origami}
The work of Dacorogna, Marcellini and Paolini on origami constructions
\cite{DMP08a,DMP08b,DMP08c,DMP10} illustrates that in some cases it is possible to saturate the sharp regularity threshold originating from scaling laws by convex integration solutions: Considering the energy $E_{\epsilon,1}$, in their case, it is possible to approach $\mu \rightarrow \frac{1}{2}$ and self-similar convex integration solutions of BV regularity exist. This shows that in their very flexible setting of the inclusion $\nabla u \in O(2)$ (or also $\nabla u \in O(n)$ and zero boundary conditions) the only obstruction is the presence or absence of trace estimates.
\end{rmk}

\section{Flexibility}
\label{sec:flex}

In this section, we recall the convex integration results from \cite{RZZ16}, \cite{RZZ17} in the model case of the geometrically linearised hexagonal-to-rhombic phase transformation (which is discussed in Section \ref{sec:hex_rhombic}). In particular, after recalling the general outline of the convex integration algorithm in Section \ref{sec:schemes}, we discuss the two different convex integration schemes, which were used in \cite{RZZ16} and \cite{RZZ17} to derive the higher regularity results (c.f. Sections \ref{sec:replace}-\ref{sec:cover}). In Section \ref{sec:numerics} we will then discuss (non-quantitative) numerical implementations of the schemes and compare these.

\subsection{The geometrically linearised hexagonal-to-rhombic phase transformation}
\label{sec:hex_rhombic}

We recall the model and some properties of the geometrically linearised hexagonal-to-rhombic phase transformation. 

\begin{figure}[t]
\includegraphics[width= 0.9\textwidth, page=1]{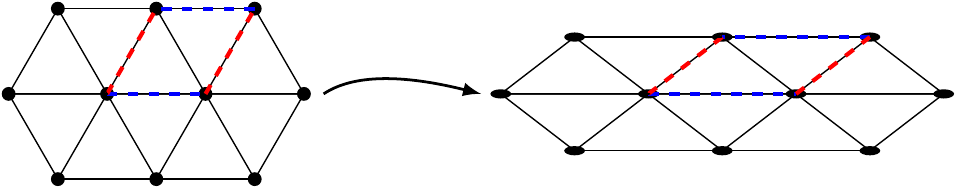}
\caption{A microscopic derivation of the hexagonal-to-rhombic phase transformation: A hexagonal lattice (which can also be interpreted to consist of a rhombic lattice) is deformed by stretching/compressing the sides of the rhombus. Carrying this out while preserving the volume and linearising the resulting deformation gradients leads to the (infinitesimal) transformation strains $e^{(1)}, e^{(2)}, e^{(3)}$.}
\label{fig:hex}
\end{figure}

A difficulty of the variational model from Section \ref{sec:BJ} is the two-fold ``nonlinearity" of the problem, manifested in the two physical requirements (i), (ii). This is also still reflected in the $m$-well problem from Section \ref{sec:m_well}, in that the set $K(\theta)$ still displays $SO(3)$ invariance, and has a multi-well structure at temperatures below $\theta_c$. Hence, in order to simplify this, it is often convenient to ``linearise" the frame indifference assumption and hence to pass from $SO(3)$ to $Skew(3)$ invariance (we recall that $Skew(3)$ is the linearisation of $SO(3)$ at the identity). Mathematically, this has the advantage of dealing with an invariance which is given by a vector space structure in contrast to a nonlinear group structure, while at the same time preserving the ``material nonlinearity", i.e. the multi-well structure. For one-well problems this has been rigorously justified in \cite{DNP02, ABK15}. In the small rotation regime, it is expected that this linearisation still captures important features of the nonlinear original problem \eqref{eq:m_well} (although care is required, in particular in the large rotation regime, c.f. \cite{Bhat93}). 

In the sequel, we study such a geometrically linearised problem in two-dimensions. Passing formally from the deformation $y(x)$ to the displacement $u(x)=y(x)-x$ and dropping all ``higher order terms" (c.f. the discussion in \cite[Chapter 11]{B}), this turns the nonlinear problem \eqref{eq:m_well} (at fixed temperature $\theta<\theta_c$) into a \emph{linearised $m$-well problem} of the type
\begin{align}
\label{eq:lin_m_well}
e(\nabla u):= \frac{1}{2}(\nabla u + (\nabla u)^t) \in \{e^{(1)},\dots,e^{(m)}\} + Skew(3).
\end{align}
Here $e(\nabla u)$ denotes the \emph{(infinitesimal) strain tensor} and $e^{(1)}, \dots, e^{(m)} \in \R^{3\times 3}_{sym}$ represent the
variants of martensite. In the sequel, we specify the transformation to be the two-dimensional \emph{hexagonal-to-rhombic} phase transformation (which is related to transformations occurring in materials such as Mg$_2$Al$_4$Si$_{18}$, Mg-Cd alloys or Pb$_3$(VO$_4$)$_2$, see \cite{KK91}, \cite{CPL14}), c.f. Figure \ref{fig:hex} for a microscopic two-dimensional derivation of the deformation matrices (the third direction can be ignored as the material only undergoes an affine change there). Assuming that the material undergoes at most an affine deformation in the third direction, allows us to reduce the three-dimensional problem to a two-dimensional one. Hence, in the following we study the differential inclusion
\begin{align}
\label{eq:hex_rhom}
e(\nabla u):= \frac{1}{2}(\nabla u + (\nabla u)^t) \in \{e^{(1)},e^{(2)},e^{(3)}\} + Skew(2),
\end{align}
where now $u:\Omega \subset \R^2 \rightarrow \R^2$ and 
\begin{align}
\label{eq:hex_rhom_wells}
e^{(1)} = \begin{pmatrix} 1 & 0 \\ 0 & -1 \end{pmatrix}, \ 
e^{(2)} = \frac{1}{2}\begin{pmatrix} -1 & \sqrt{3} \\ \sqrt{3} & 1 \end{pmatrix}, \ 
e^{(3)} = \frac{1}{2}\begin{pmatrix} -1 & -\sqrt{3} \\ -\sqrt{3} & 1 \end{pmatrix}.
\end{align}
Note that the traces of all three matrices in \eqref{eq:hex_rhom_wells} vanish, which corresponds to the modelling assumption that the transformation is (infinitesimally) volume preserving.

We recall two important features of the model \eqref{eq:hex_rhom}, \eqref{eq:hex_rhom_wells}, which were discussed in more detail in Section 2 in \cite{RZZ16} (c.f. also Figure \ref{fig:rank_one}):

\begin{lem}
\label{lem:rank_one_conn}
Let $e^{(1)}, e^{(2)}, e^{(3)}$ be as in \eqref{eq:hex_rhom_wells}. Then, for each pair $i,j \in \{1,2,3\}$, with $i\neq j$, there exist (up to a change of sign) exactly two vectors $n_{ij}^{1}, n_{ij}^{2}\in S^{1}$ and two vectors $a_{ij}^{1}, a_{ij}^{2} \in \R^2 \setminus \{0\}$ such that 
\begin{align*}
e^{(i)}-e^{(j)} = \frac{1}{2}(a_{ij}^{k}\otimes n_{ij}^{k} + n_{ij}^k \otimes a_{ij}^k)=: a_{ij}^k\odot n_{ij}^k, \ k \in \{1,2\}. 
\end{align*}
\end{lem}

\begin{figure}[t]
\includegraphics[width=4cm, page=2]{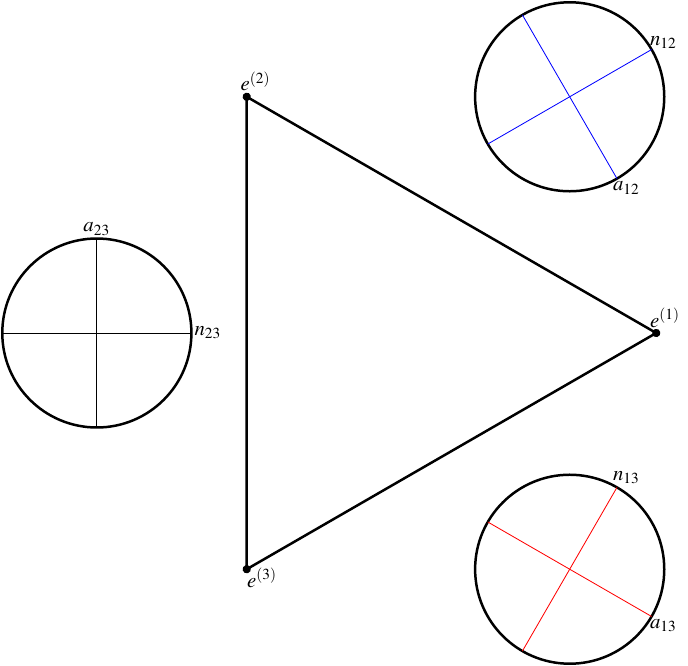}
\caption{
The symmetrised rank-one connections between the wells: The triangle depicts the equilateral triangle that is spanned in strain space by the wells $e^{(1)}, e^{(2)}, e^{(3)}$. As stated in Lemma \ref{lem:rank_one_conn} these are pairwise symmetrised rank-one connected with two possible rank-one connections each (due to the trace free constraint these are related by $\frac{\pi}{2}$-rotations). The orientation of these rank-one directions are depicted in the circles which are drawn next to the corresponding sides of the equilateral triangle. Here the choice of the sign of the normals is still free, depending on the ordering of the wells. }
\label{fig:rank_one}
\end{figure}

This lemma in particular implies that for every pair $e^{(i)}, e^{(j)}$ with $i \neq j$ there exist \emph{simple laminates}. These are one-dimensional microstructures in which the strain only attains the values $e^{(i)}, e^{(j)}$, i.e. there exists $\nu \in S^1$ such that
\begin{align*}
e(\nabla u)(x) = f(x \cdot \nu) \mbox{ and } f(x\cdot \nu) \in \{e^{(i)}, e^{(j)}\}. 
\end{align*}
Moreover, (up to a change of sign) $\nu \in \{n_{ij}^1, n_{ij}^2\}$. These are however not the only possible piecewise affine microstructures, which can occur in the hexagonal-to-rhombic phase transformation. For a complete list of ``homogeneous" deformations we refer to Section 7 in \cite{RZZ16}. Moreover, we remark that these ``homogeneous" deformations can also be concatenated into more complex microstructures.

As a consequence of the presence of the rank-one connections between the wells, the transformation displays a very flexible behaviour. This is manifested in the size of its convex hulls (c.f. Figure \ref{fig:lam_hull}):

\begin{lem}
\label{lem:hulls}
Let $e^{(1)}, e^{(2)}, e^{(3)}$ be as in \eqref{eq:hex_rhom_wells}. Then,
\begin{align*}
\{e^{(1)}, e^{(2)}, e^{(3)}\}^{lc} = 
\conv(\{e^{(1)}, e^{(2)}, e^{(3)}\}).
\end{align*}
In particular, 
\begin{align*}
\dim(\{e^{(1)}, e^{(2)}, e^{(3)}\}^{lc}) = 
\dim(\conv(\{e^{(1)}, e^{(2)}, e^{(3)}\}))=2.
\end{align*}
\end{lem}

\begin{figure}[t]
\includegraphics[width=4cm, page=3]{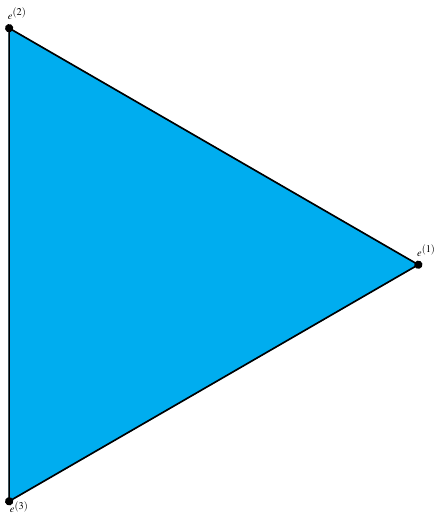}
\caption{The laminar convex hull of the wells $e^{(1)}, e^{(2)}, e^{(3)}$. Using two-fold laminations, it is possible to reach any boundary datum in the interior of the convex hull of the three wells $e^{(1)}, e^{(2)}, e^{(3)}$.
}
\label{fig:lam_hull}
\end{figure}

For a discussion of the argument leading to this result we refer to Section 2 in \cite{RZZ16} and \cite[Chapter 11]{B}. In particular, in full matrix space, the inclusion \eqref{eq:hex_rhom} is a co-dimension one problem (the only constraint coming from the trace condition). Hence, there is enough room to construct convex integration solutions:

\begin{thm}[Existence of convex integration solutions]
\label{prop:conv_int}
Let $\Omega \subset \R^2$ be an open Lipschitz domain.
Let $e^{(1)},e^{(2)},e^{(3)}$ be as in \eqref{eq:hex_rhom_wells}. Then for any $M\in \R^{2\times 2}$ with
$$e(M):=\frac{1}{2}(M+M^t) \in \intconv(\{e^{(1)},e^{(2)},e^{(3)}\})$$ 
there exists a solution $u$ of \eqref{eq:hex_rhom}. Moreover, it is possible to ensure that for any $\epsilon>0$ there is a solution $u_{\epsilon}$ to \eqref{eq:hex_rhom} with
\begin{align}
\label{eq:approx}
\|u_{\epsilon}-Mx\|_{L^{\infty}(\Omega)}< \epsilon.
\end{align}
\end{thm}

This (non-quantitative) result follows for instance from the arguments of Müller-{\v{S}}ver{\'a}k \cite{MS} or from the results of Dacorogna-Marcellini \cite{DM1}. As seen from the approximation result \eqref{eq:approx} these solutions are highly non-unique; in \cite{DM1} this is made more precise in a Baire category sense.
We remark that in invoking the arguments from \cite{MS} and \cite{DM1}, there is a slight subtlety here, in that the inclusion is formulated on the level of the symmetrised gradient and not on the level of the full gradient. However, there are various ways of overcoming this, one being to simply ``pull up" the inclusion problem to an inclusion problem for the full gradient.

In \cite{RZZ16} and \cite{RZZ17} together with B. Zwicknagl we analysed the underlying construction schemes more precisely and showed that it is possible to improve the regularity of solutions to \eqref{eq:hex_rhom} on a ($W^{s,p}$) Sobolev scale by choosing the underlying lengths scales carefully. The iterative algorithms in \cite{RZZ16} and \cite{RZZ17} however differed quantitatively, in that they produced different dependences: While the regularity of solutions constructed by the scheme from \cite{RZZ16} depended on the position of the boundary data $M$ in matrix space, the scheme from \cite{RZZ17} could produce solutions with a regularity which did not depend on this (or more precisely, where we could give bounds which were independent of this). 

In the sequel, we discuss and compare explicit numerical implementations of these schemes, which illustrate the differences which occur here. While convex integration solutions in the context of elasticity have been known for quite a while, the present note seems to contain the first numerical implementation of these. In our numerics, we focus on the \emph{qualitative} convex integration result from Theorem \ref{prop:conv_int} and do not seek to optimize the underlying partitions, which was necessary in the quantitative analysis in \cite{RZZ16}, \cite{RZZ17}. This is due to the introduction of highly fractal structures which would exceed the capability of our computers (and which from a certain level onwards would also not be seen as major changes without zooming into the structures). Also, we believe that, while being analytically convenient, the exact coverings from \cite{RZZ16}, \cite{RZZ17} are far from optimal. Since the numerical implementation without this already provides interesting insights, we opted to focus on the problems without the additional layer of (analytical and numerical) difficulty originating from the complicated covering structures from \cite{RZZ16}, \cite{RZZ17}.

\subsection{The outline of the convex integration scheme for the hexagonal-to-rhombic phase transformation}
\label{sec:schemes}

We recall the convex integration schemes from \cite{RZZ16}, \cite{RZZ17} applied to the hexagonal-to-rhombic phase transformation. In the following section, we then discuss and compare some numerical implementations of these. 

We begin by formulating the rough outline of the convex integration algorithms which are used in \cite{RZZ16} and \cite{RZZ17} in terms of ``pseudo-code". In their main structure the two algorithms are similar, however there are important differences which are mainly encoded in the functions, which are used in Step (1b).

\begin{alg}
\label{alg:conv_int}
Let $M \in \mathcal{M}:=\{N \in \R^{2\times 2}: \ e(N)\in \intconv(\{e^{(1)}, e^{(2)}, e^{(3)}\})\}$, where $e^{(1)}, e^{(2)}, e^{(3)}$ are as in \eqref{eq:hex_rhom_wells}. Let $\Omega \subset \R^2$ be a triangle.
\begin{itemize}
\item[(1a)] \emph{Variables.} We consider 
\begin{itemize}
\item the displacement $u_k: \Omega \rightarrow \R^2$ at step $k$, 
\item a collection of (up to null-sets disjoint) triangles $\widehat{\Omega}_k = \{\Omega_1^k,\dots,\Omega_{j_k}^k\}$, which cover $\Omega$, 
\item and the error in matrix space $\epsilon_k: \widehat{\Omega}_k \rightarrow  (0,1)$ at step $k$, which is constant on each subset of $\widehat{\Omega}_k$.
\end{itemize}
\item[(1b)] \emph{Functions.} 
We consider a ``covering function", which covers a given triangle by ``good sets", on which the deformation $u_k$ will be improved, and a ``remainder". More precisely,
\begin{align*}
\Cover_{v}: &\mathcal{T} \times \mathcal{M} \rightarrow \mathcal{R} \times \mathcal{T},\\
&(T,M) \mapsto (\{Q_1,\dots,Q_{j(M,T)}\} , \{T_1,\dots, T_{l(M,T)}\}).
\end{align*}
Here $\mathcal{T}$ is the set of all triangles, $\mathcal{R}$ is a set of certain quadrilaterals (these differ in \cite{RZZ16} and \cite{RZZ17}) and $\mathcal{M}:=\{M\in \R^{2\times 2}: \ e(M) \in \intconv(\{e^{(1)}, e^{(2)}, e^{(3)}\}) \}$. For each $T \in \mathcal{T}$ and $M \in \mathcal{M}$ we have an (up to null sets) disjoint covering
\begin{align*}
T = \bigcup\limits_{Q \in \Cover_{v}(T,M)[1]} Q \cup \bigcup\limits_{\tilde{T}\in\Cover_{v}(T,M)[2]}\tilde{T}, \ 
\end{align*}
with
\begin{align*}
\frac{\left|\bigcup\limits_{Q \in \Cover_{v}(T,M)[1]} Q \right|}{|T|} \geq v.
\end{align*}
The sets $Q \in \Cover_v(T,M)[1]$ are the ``good sets", on which the current displacement gradient $\nabla u_k$ will be modified and pushed towards the energy wells. The sets $\tilde{T} \in \Cover_v(T,M)[2]$ are the ``remainders" on which the displacement $u_k$ is not changed.

We further consider a ``replacement function", which improves the current displacement gradient in the sense that the replaced deformation gradient is closer to the wells (or at least closer to the wells on a large portion of the domain). This depends on the current displacement gradient, the underlying domain and the error in matrix space. As an output it yields 
\begin{itemize}
\item the level sets (in the form of a finite collection of triangles) of the new improved deformation, 
\item a piecewise affine function whose symmetric gradient attains values in $\intconv(\{e^{(1)}, e^{(2)}, e^{(3)}\})$,
\item and an updated error in matrix space. 
\end{itemize}
More precisely,
\begin{align*}
\Replace: \ &\mathcal{R} \times \mathcal{A}_{\intconv(\{e^{(1)}, e^{(2)}, e^{(3)}\})} \times (0,1) \\
& \quad \rightarrow \mathcal{T}  \times \mathcal{A}_{\intconv(\{e^{(1)}, e^{(2)}, e^{(3)}\})} \times (0,1),\\
& \quad (Q,w_{\text{old}},\epsilon) \mapsto (\{T_{1},\dots,T_{j_0}\} , w, \tilde{\epsilon})
\end{align*}
Here $\mathcal{A}_{\intconv(\{e^{(1)}, e^{(2)}, e^{(3)}\})}$ denotes the set of piecewise affine deformations with symmetric gradients in $\intconv\{e^{(1)},e^{(2)},e^{(3)}\}$. Furthermore, $w|_{\partial Q}=w_{old}|_{\partial Q}$.
\item[(2)] \emph{Initialization.} We begin by setting $u_0(x)=Mx$, $\epsilon_0 = \epsilon_0(M)$, $\widehat{\Omega}_0 = \{\Omega\}$. 
\item[(3)] \emph{Iteration step.} The algorithm proceeds iteratively: Assume that $u_k$, $\epsilon_k$ and $\widehat{\Omega}_k$ with $k\geq 0$ are already given.
Then on each $\Omega_{j}^k \subset \widehat{\Omega}_k$ for which  $e(\nabla u_k)|_{\Omega_j^k} \notin \{e^{(1)}, e^{(2)}, e^{(3)}\}$, 
apply the function $\Cover_v(\Omega_{j}^k, \nabla u_k|_{\Omega_j^k})$. Let $\Cover_v(\Omega_{j}^k,\nabla u_k|_{\Omega_j^k})[1]= \{\Omega_{j,1}^k,\dots,\Omega_{j,m(j,k)}^k\}$.

For each $\Omega_{j,m}^k \in \Cover_v(\Omega_j^k,\nabla u_k|_{\Omega_j^k})[1]$ apply the function \\
$\Replace(\Omega_{j,m}^k, u_k |_{\Omega_{j}^k},\epsilon_k)$. This yields 
\begin{itemize}
\item[(i)] an up to null-sets disjoint covering of $\Omega_{j,m}^k$ into triangles 
$$\{\Omega_{j,m,1}^k,\dots,\Omega_{j,m,l(j,m,k)}^k\}:=\Replace(\Omega_{j,m}^k, u_k|_{\Omega_{j}^k},\epsilon_k)[1];$$
\item[(ii)] a function $v_{j,m,k}:=\Replace(\Omega_{j,m}^k, u_k|_{\Omega_{j}^k},\epsilon_k)[2]: \Omega_{j,m}^k \rightarrow \R^2$ whose gradient is constant on each of the sets $\Omega_{j,m,l}^k$ with $l\in \{1,\dots, l(j,m,k)\}$ and for which $v_{j,m,k}(x)=u_k(x) \mbox{ for all } x \in \partial \Omega_{j,m}^k$;
\item[(iii)] a parameter $\tilde{\epsilon}_{j,m,k}:=\Replace(\Omega_{j,m}^k,  u_k |_{\Omega_{j}^k},\epsilon_k)[3]$.
\end{itemize}
\noindent We then set
\begin{align*}
&u_{k+1}|_{\Omega_{j,m}^k}=v_{j,m,k} \mbox{ and } u_{k+1}|_{\Cover_v(\Omega_j^k,\nabla u_k|_{\Omega_j^k})[2]}=u_k|_{\Cover_v(\Omega_j^k,\nabla u_k|_{\Omega_j^k})[2]},\\
&\widetilde{\Omega}_{k+1} = \bigcup\limits_{j,m} \Replace(\Omega_{j,m}^k,  u|_{\Omega_{j}^k},\epsilon_k)[1] \cup \bigcup\limits_{j} \Cover_v(\Omega_j^k,\nabla u_k|_{\Omega_j^k})[2],\\
& \epsilon_{k+1}|_{T} = \left\{
\begin{array}{ll}
\tilde{\epsilon}_{j,m,k} &\mbox{ if } T \subset \Omega_{j,m,l}^k \in \Cover_v(\Omega_j^k,\nabla u_k|_{\Omega_j^k})[1] \mbox{ for some } j \in \{1,\dots,j_k\},\\
\epsilon_{k}|_{T} &\mbox{ if } T \in \Cover_v(\Omega_j^k,\nabla u_k|_{\Omega_j^k})[2] \mbox{ for some } j \in \{1,\dots,j_k\}.
\end{array}
\right.
\end{align*}
If on $\Omega_{j}^k$ we already have $e(\nabla u_k)|_{\Omega_{j}^k} \in \{e^{(1)}, e^{(2)}, e^{(3)}\}$, we set $u_{k+1}|_{\Omega_j^k} = u_{k}|_{\Omega_j^k}$ and
\begin{align*}
\widehat{\Omega}_{k+1}=\widetilde{\Omega}_{k+1}\cup \{\Omega_{j}^{k}\in \widehat{\Omega}_k: \nabla u_k|_{\Omega_{j}^{k}}\in\{e^{(1)}, e^{(2)},e^{(3)}\} \mbox{ a.e.}\}.
\end{align*}
\end{itemize}
\end{alg}

It has been shown in \cite{RZZ16} and \cite{RZZ17} that this procedure is well-defined, if the function $\Replace$ is chosen appropriately (there is a slight subtlety in that in those articles we approximate domains by rectangles, but this does not matter for the non-quantitative algorithm which is used here).
In spite of their similar overall structure, the algorithms from \cite{RZZ16}, \cite{RZZ17} however differ substantially in their underlying replacement constructions encoded by the function $\Replace$. We discuss the different constructions for the function $\Replace$ in Sections \ref{sec:conv_int_1} and \ref{sec:conv_int_2}. In Section \ref{sec:cover} we also present the details on the function $\Cover_v$.

\subsection{The replacement constructions}
\label{sec:replace}

As they form the core of the convex integration algorithm, we discuss the replacement constructions, which are encoded in the function $\Replace$ in the Algorithm \ref{alg:conv_int}, in more detail. Here we used different options in \cite{RZZ16} and \cite{RZZ17}. In spite of this, the main aim of the replacement construction is the same: Given a prescribed gradient distribution, we seek to modify and improve it in the sense that it becomes ``closer" to being an exact solution of the differential inclusion \eqref{eq:hex_rhom}, \eqref{eq:hex_rhom_wells}.\\
On the one hand, in the article \cite{RZZ16}, we achieved this by a ``piecewise affine" construction in which for a given point $x\in \Omega$ the gradient distribution ended up being exactly in the well after a finite (but $x$-dependent) number of steps. However, on parts of the domain, the gradient distribution deteriorated, in the sense that on a part of the domain, the gradient was pushed away from the wells.
On the other hand, in \cite{RZZ17} in general, we had to modify the gradient distribution for each point countably many times, but improved the overall gradient distribution in each step essentially in a ``uniform way". We describe this in more detail in the following two sections. 

\subsubsection{The replacement construction from \cite{RZZ16}}
\label{sec:conv_int_1}

In \cite{RZZ16} the replacement construction (i.e. the function $\Replace$) is based on the following lemma, which is a modification of a construction due to Conti \cite{C} (c.f. Figure \ref{fig:rectangle}):

\begin{lem}[Lemma 3.5 in \cite{RZZ16}]
\label{lem:replace1}
Let $M\in \R^{2\times 2}$ with 
\begin{align*}
e(M):= \frac{1}{2}(M+M^t) \in \intconv(K)
\end{align*}
and let $e^{(i)}$ with $i\in\{1,2,3\}$ (as in \eqref{eq:hex_rhom_wells}) be such that
\begin{align}
\label{eq:nearest_well}
|e(M)-e^{(i)}| \leq \dist(e(M),K) + 4 \epsilon_0,
\end{align}
where $\epsilon_0\in (0,\dist(e(M), \partial\conv(K)))/100$. 
Then there exists a rectangular domain $\Omega_{\delta}(M)$ (which in general is
rotated with respect to the coordinate axes) and a Lipschitz function $u:\R^2
\rightarrow \R^2$, with symmetric gradients $e^{(i)} \in K, \tilde{e}_1, \dots, \tilde{e}_4 \in \intconv(K)$ such that 
\begin{itemize}
\item[(i)] $e(\nabla u) \in \{e^{(i)}, \tilde{e}_1, \dots, \tilde{e}_4\} \subset \conv(K) \mbox{ in } \Omega_{\delta}(M)$ and
\begin{align*}
|\{x\in \Omega_{\delta}(M): e(\nabla u)(x)=e^{(i)}\}|= \frac{1}{4}|\Omega_{\delta}(M)|.
\end{align*}
\item[(ii)] $u(x)=Mx \mbox{ on } \R^2 \setminus \Omega_{\delta}(M)$.
\end{itemize}
\end{lem}

\begin{figure}[t]
\includegraphics[width=7cm, page=4]{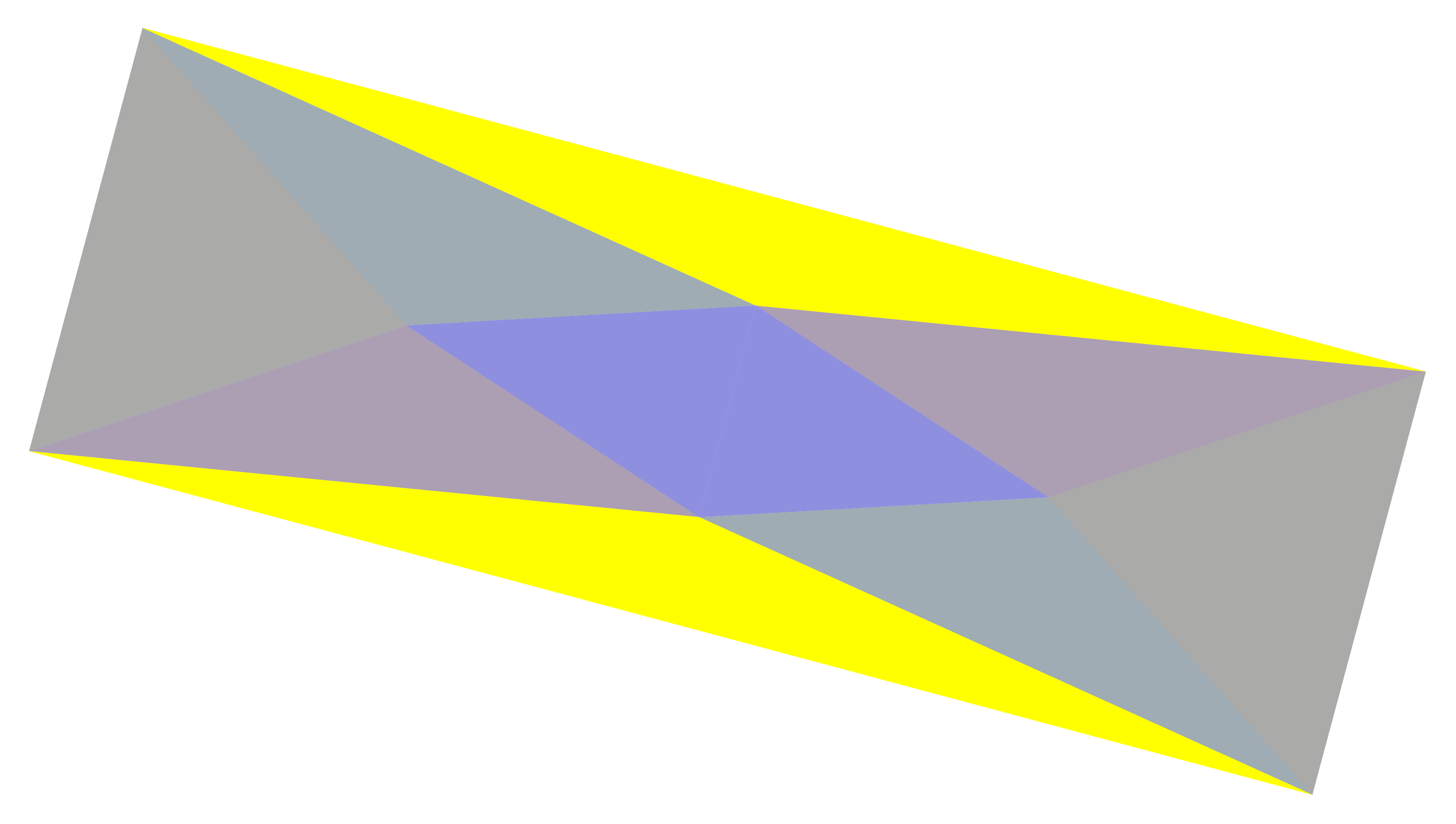}
\caption{
The replacement construction from Lemma \ref{lem:replace1}. The triangles correspond to the level sets of the replacement function (with different colours corresponding to different values of the symmetrised gradient). }
\label{fig:rectangle}
\end{figure}

The distribution of the symmetric gradients $\tilde{e}_1,\dots, \tilde{e}_4$ in strain space is illustrated in Figure \ref{fig:repl_colours}: On a fixed volume fraction of the domain the modified gradient is pushed exactly into one of the wells, but on the other parts of the domain it either remains close to the original boundary datum, or is pushed even further into the interior of the equilateral triangle spanned by $e^{(1)}, e^{(2)}, e^{(3)}$.

\begin{figure}[t]
\includegraphics[width=7cm, page=5]{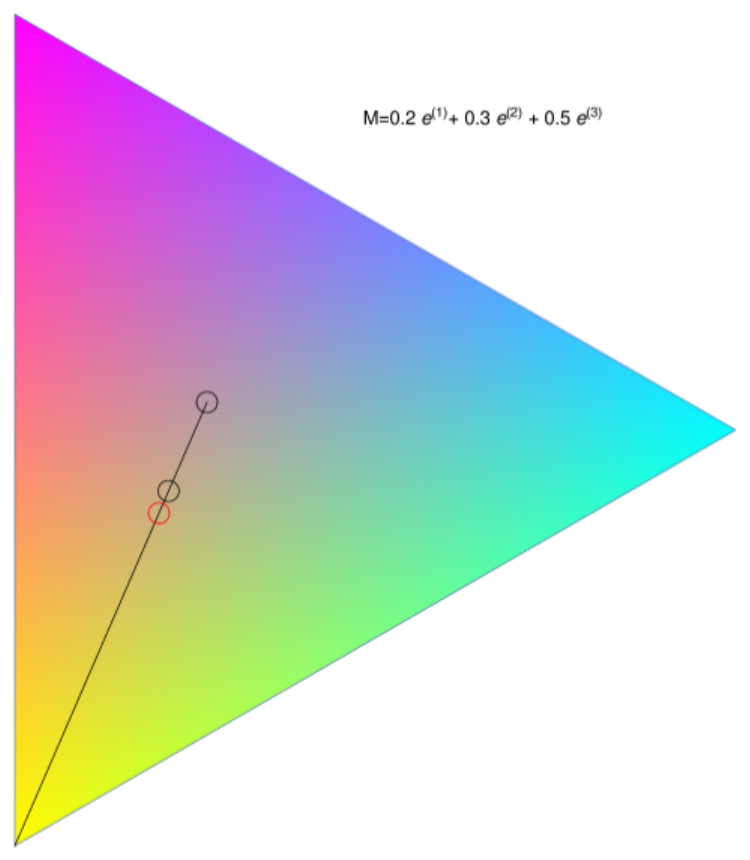}
\caption{
The replacement construction from \cite{RZZ16} depicted in CMY colour encoding. The colour triangle corresponds to the wells, in that $e^{(1)}$ is depicted as cyan, $e^{(2)}$ as magenta and $e^{(3)}$ as yellow. Correspondingly convex combinations are presented as blends of these colours. The algorithm from \cite{RZZ16} starts with a given boundary datum (here in the red circle) and replaces it by five new matrices, one exactly in the wells (here the yellow well), and four in the interior (one in the red and three further ones in the two black circles). The size of the circles represents the deviation from the originally chosen rank-one line. In particular, we note that some of the new data are ``pushed" further into the interior of the convex hull of $e^{(1)}, e^{(2)}, e^{(3)}$.}
\label{fig:repl_colours}
\end{figure}

In the context of Algorithm \ref{alg:conv_int} this implies the following: Given
the current displacement $u_k$, we modify it on each of its level sets
$\Omega_j^k$ by first covering $\Omega_j^k$ by finitely many (translations and
rescalings) of the domains $\Omega_{\delta}(\nabla u_k|_{\Omega_{j}^k})$ from
the Lemma \ref{lem:replace1} and a remainder, which consists of finitely many
triangles (this can always be achieved, c.f. the discussion of the function
$\Cover_v$ in Section \ref{sec:cover}). The number of domains in the covering is
chosen such that a sufficiently large volume fraction $v\in (0,1)$ is covered
($v=1$ would be ideal, but in general, this would require countably many sets, which is not
numerically feasible; in the quantitative convex integration algorithm from \cite{RZZ16} we also excluded this, as it would yield infinite surface energy in general). On the (translations and rescalings of the) domains $\Omega_{\delta}(\nabla u_k)$ we invoke Lemma \ref{lem:replace1}. 

In particular, the construction from Lemma \ref{lem:replace1} is such that the displacement $u_k$ is replaced by a displacement $u_{k+1}$ whose symmetric gradient is already exactly contained in one of the wells on a subset of $\Omega_{\delta}(\nabla u_k)$ of a given volume fraction (which is in particular uniformly bounded from below, e.g. by $\frac{1}{8}$). This set is no longer touched in the remainder of the Algorithm \ref{alg:conv_int}. Hence, the volume fraction of the original domain, on which the displacement gradient is modified in the $k$-th step, is bounded by $\left(1-\frac{v}{8}\right)^k$. In particular, for almost any point $x\in \Omega$ there exists an integer $k(x)$ such that the gradient $\nabla u_k(x)$ is already exactly in the wells after $k(x)$ iteration steps. 
On the complement of this set, i.e. on the set, where the displacements are not yet in the energy wells, we push a part of the gradient distribution further into the interior of $\intconv\{e^{(1)}, e^{(2)},e^{(3)}\}$ and on another part we remain in an $\epsilon_k$ neighbourhood of the original $e(\nabla u_k)|_{\Omega_{j}^k}$. In particular, on these parts of the domain, we either \emph{do not} improve or even deteriorate the underlying displacement construction (c.f. Figure \ref{fig:repl_colours}).
This is iterated.

Due to the ``pushing into the interior", after a finite number of steps, the closest well changes (i.e. the choice of $e^{(i)}$ in \eqref{eq:nearest_well} changes). This corresponds to a change in the rank-one direction in matrix space, which in turn is reflected in the orientation of the underlying domains $\Omega_{\delta}(\nabla u_k)$. In particular, here any possible orientation (in the union of certain cones, c.f. the blue cones in Figure 3 in \cite{RZZ16}) may arise in general, they are not prescribed by the rank-one directions (c.f. Lemma \ref{lem:rank_one_conn}) of the energy wells $\{e^{(1)}, e^{(2)}, e^{(3)}\}$. \\

The key features of the replacement construction in \cite{RZZ16} can hence be summarized as follows:
\begin{itemize}
\item It is ``finite", in the sense that for almost every point $x\in \Omega$ there exists $k(x)\in \N$ such that the iteration creates a gradient distribution for which $e(\nabla u_k)(x)$ already attains one of the values $e^{(1)}, e^{(2)}$ or $e^{(3)}$ after a finite number $k(x)$ of steps. After this has been achieved the gradient distribution at that point will not be altered anymore.
\item It is ``non-uniform", in that by using the ``pushing out" construction the gradient distribution is deteriorated on part of the domain (i.e. moved further away from the wells) in each iteration step.
\item It possibly involves a continuum of possible orientations and normals.
\end{itemize}

\begin{figure}[t]
\begin{center}
\begin{tabular}{p{1.5cm}|p{3cm} p{3cm} p{3cm}}
 & Number of iterations & Local convergence  & Number of directions\\
\hline
\cite{RZZ16} & finite depending on point & exactly attained, after non-uniform number of steps; partially pushed into the interior & continuum of directions in cones\\
\cite{RZZ17} & countably infinite & uniform, exponential convergence rate & finitely many directions\\
\hline
\end{tabular}
\end{center}
\caption{A comparison of the properties of the two replacement functions described above.}
\label{tab:compare}
\end{figure}

\subsubsection{The replacement construction from \cite{RZZ17}}
\label{sec:conv_int_2}

In the algorithm from \cite{RZZ17} the replacement construction is based on the following observation:

\begin{lem}[Lemmas 5.22 and 5.23 in \cite{RZZ17}]
\label{lem:replace2}
Let $M \in \R^{2\times 2}$ be such that $e(M)\in \intconv(K)$. Assume that
$\dist(M,K)=d_0$. Then, for a constant $ \gamma \in (0,1)$, there exists a domain $\Omega_{\delta}(M)$, which is diamond-shaped (but possibly rotated), and a piecewise affine map $w:\Omega_{\delta}(M) \rightarrow \R^2$ such that
\begin{itemize}
\item[(i)] $\dist(\nabla w, K)\leq \gamma d_0$,
\item[(ii)] $w(x)=Mx$ on $\partial \Omega_{\delta}(M)$.
\end{itemize}
\end{lem}

\begin{rmk}
\label{rmk:uniform}
Regarding the constant $\gamma \in (0,1)$, we remark that it can be chosen in a uniform way in the sense that there exists $\sigma \in (0,1)$ (independent of $M$; almost $1/2$ up to $\epsilon$ errors) such that the distance between $\nabla w$ and $K$ improves by a factor $\sigma$, where the distance is measured in
$l^{\infty}$ with respect to the convex coefficients of $M= \lambda_{i}e^{(i)}$.
\end{rmk}

\begin{figure}[t]
\includegraphics[height=6cm, page=6]{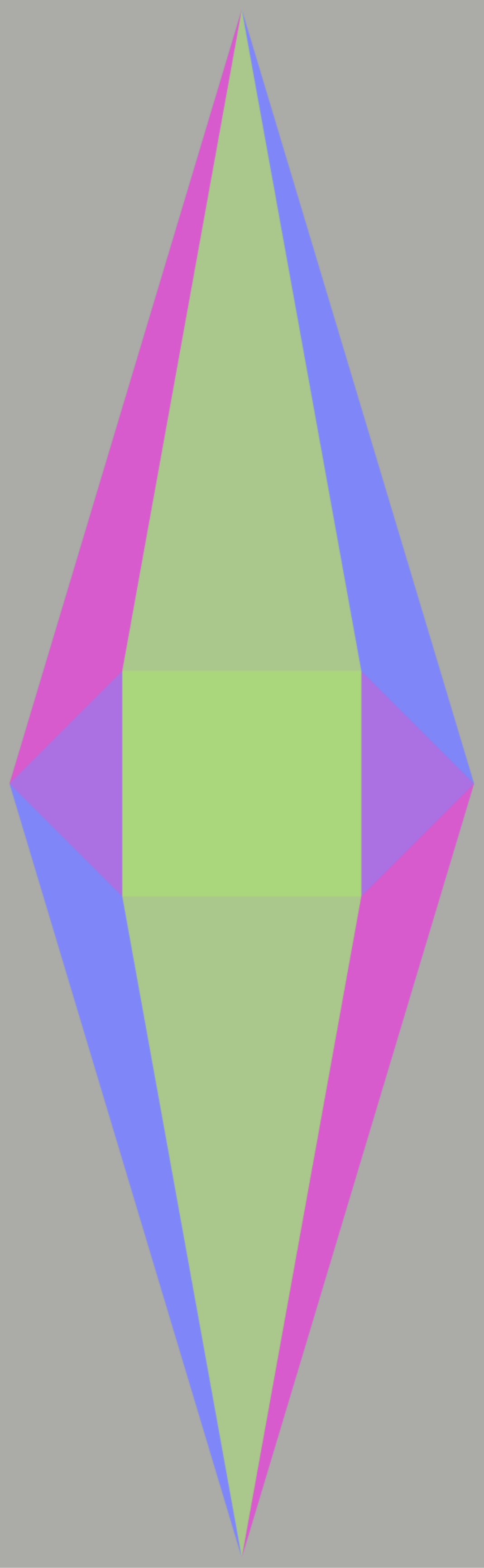}
\includegraphics[height=6cm, page=7]{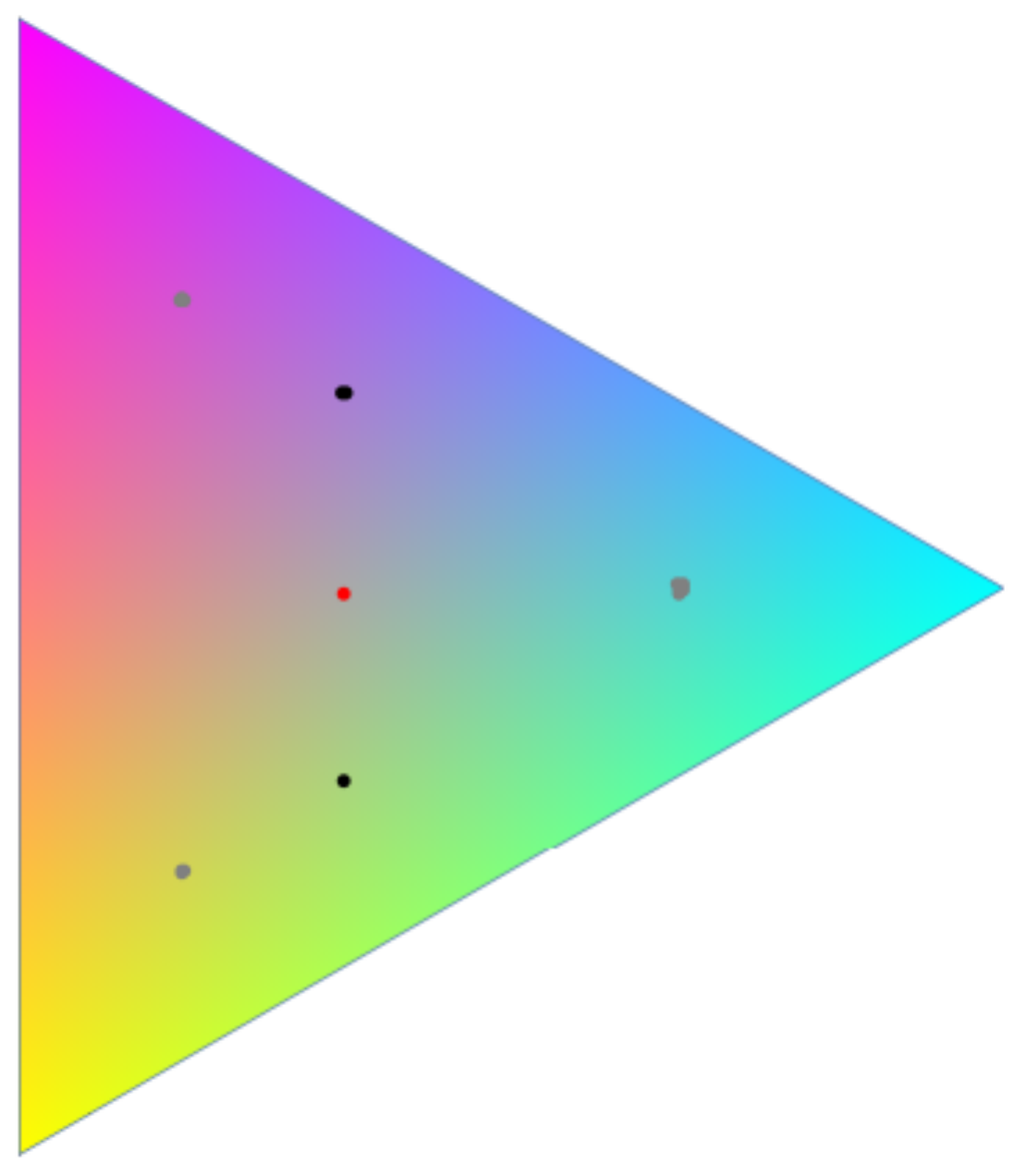}
\caption{
The replacement construction from Lemma \ref{lem:replace2}. Again the triangles
correspond to the level sets of the replacement function (with different colours
corresponding to different values of the symmetrized gradient).
On the right we have plotted the position of the new matrices with respect to the boundary datum 
in black and the level sets of a further iteration step in grey. We note the ``uniformity" of the algorithm, which is reflected in the fact that the successive iterations approach the wells uniformly (in barycentric coordinates). In order to highlight this uniformity, we have chosen a comparatively small value of the error in matrix space (significantly smaller than the one used for the left panel).}
\label{fig:diamond}
\end{figure}

The construction of Lemma \ref{lem:replace2} is illustrated in Figure \ref{fig:diamond}. We
remark that the directions along which $\nabla w$ is modified and along which $\Omega_{\delta}(M)$ is oriented correspond to the
rank-one directions from Lemma \ref{lem:rank_one_conn}. In particular, only
finitely many of these directions occur. \\

As before, in the application of Lemma \ref{lem:replace2} in Algorithm \ref{alg:conv_int}, we first
cover a sufficiently large volume fraction of the domains $\Omega_j^k$ by
finitely many translations and rescalings of the sets $\Omega_{\delta}(\nabla
u_k|_{\Omega_j^k}) \in \Cover_v(\Omega_j^k, \nabla u_k|_{\Omega_j^k})[1]$. On these we apply the
replacement construction from Lemma \ref{lem:replace2}. We note that in contrast
to the scheme from Section \ref{sec:conv_int_1}, the new piecewise affine
deformation is in general \emph{not} exactly contained in the energy wells on
any of its level sets -- not even on a part of the underlying domain. In
particular, the volume of the domain on which the displacement gradients have to
be improved in the next step, does not decrease; in the above sense this
algorithm is not ``finite" but ``countable". However, in contrast to the
previous scheme, we now improve the gradient distribution ``uniformly" and (essentially)
dyadically on all of its level sets (on which it is modified), which yields geometric convergence of the
gradient distribution. On the modified domain the gradient distribution is
always improved by a factor of (nearly) two with respect to an $l^{\infty}$-based metric on the barycentric coordinates (c.f. property (i) in Lemma \ref{lem:replace2}).\\

Hence the properties of this scheme can be summarized as follows:
\begin{itemize}
\item It is ``countable", in the sense that in general the construction does not push the gradient distribution exactly into one of the wells within a finite iteration time.
\item It is ``uniform", in that the gradient distribution is always improved by a factor two in the whole domain on which the function $\Replace$ is acting on. 
\item It involves a finite set of possible normals and orientations.
\end{itemize}

We emphasize that in both replacement constructions we obtain good control on the skew part. While this plays a role in the quantitative argument in \cite{RZZ16}, it is not of major importance in our qualitative scheme. Hence, we do not discuss this further here.

\subsection{The covering constructions}
\label{sec:cover}

As the form of the replacement function $\Replace$
has already been discussed in the previous section, it remains to describe our implementation
of the function $\Cover_v$. Given a triangle $T \in \mathcal{T}$ and a matrix $M \in \mathcal{M}$, i.e. $e(M)\in \intconv(\{e^{(1)}, e^{(2)}, e^{(3)}\})$, we essentially proceed by a greedy algorithm. We cover $T$ by a grid of dyadic scales
consisting of dyadic rescalings and translations of the domains $\Omega_{\delta}(M)$ from Lemmas \ref{lem:replace1} and \ref{lem:replace2}. If a grid parallelogram is a strict subset of $T$, we keep it and add it to the
list $\Cover_v(T, M)[1]$. If a parallelogram is a strict subset of $T^c$, we discard the parallelogram. For all parallelograms intersecting the boundary of $T$, we dyadically refine the parallelogram, and repeat the process on these refined parallelograms. 
We iterate this until we have covered a volume fraction of size at least $v \in (0, 1)$. At each step, we have a list $\Cover_v(T, M)[1]$ of ``good" parallelograms (green in Figure \ref{fig:cover}), the discarded parallelograms which do not need to be stored any longer (black in Figure \ref{fig:cover}), and a remaining list of parallelograms to be further refined (white in Figure \ref{fig:cover}). As our shapes are convex polytopes, efficient intersection algorithms can be implemented.
The remainder of the domain, which is not covered by this procedure is split into
triangles, which are inserted into $\Cover_v(T, M)[2].$

\begin{figure}[t]
\includegraphics[width=3cm, page=8]{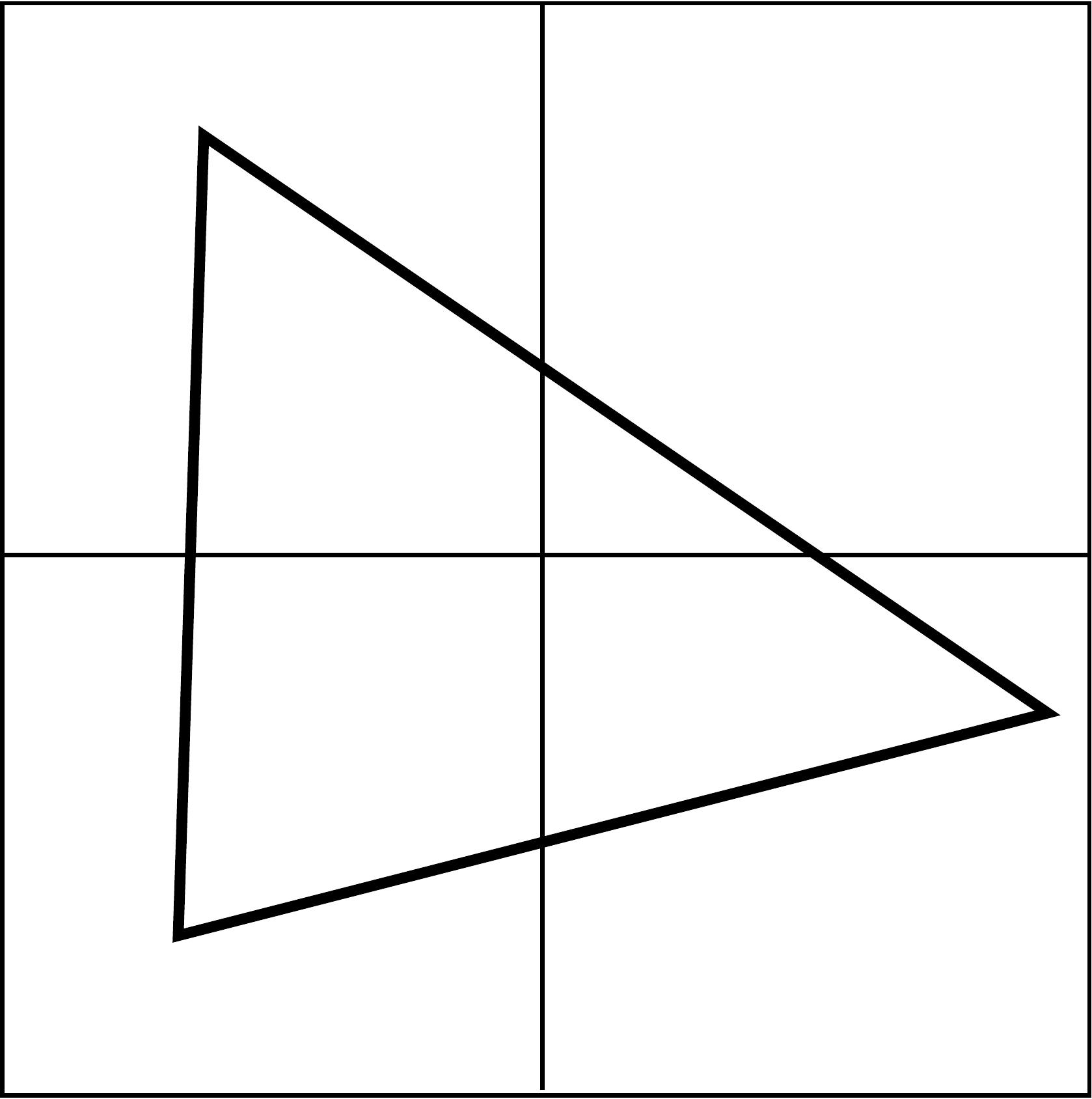}
\includegraphics[width=3cm, page=9]{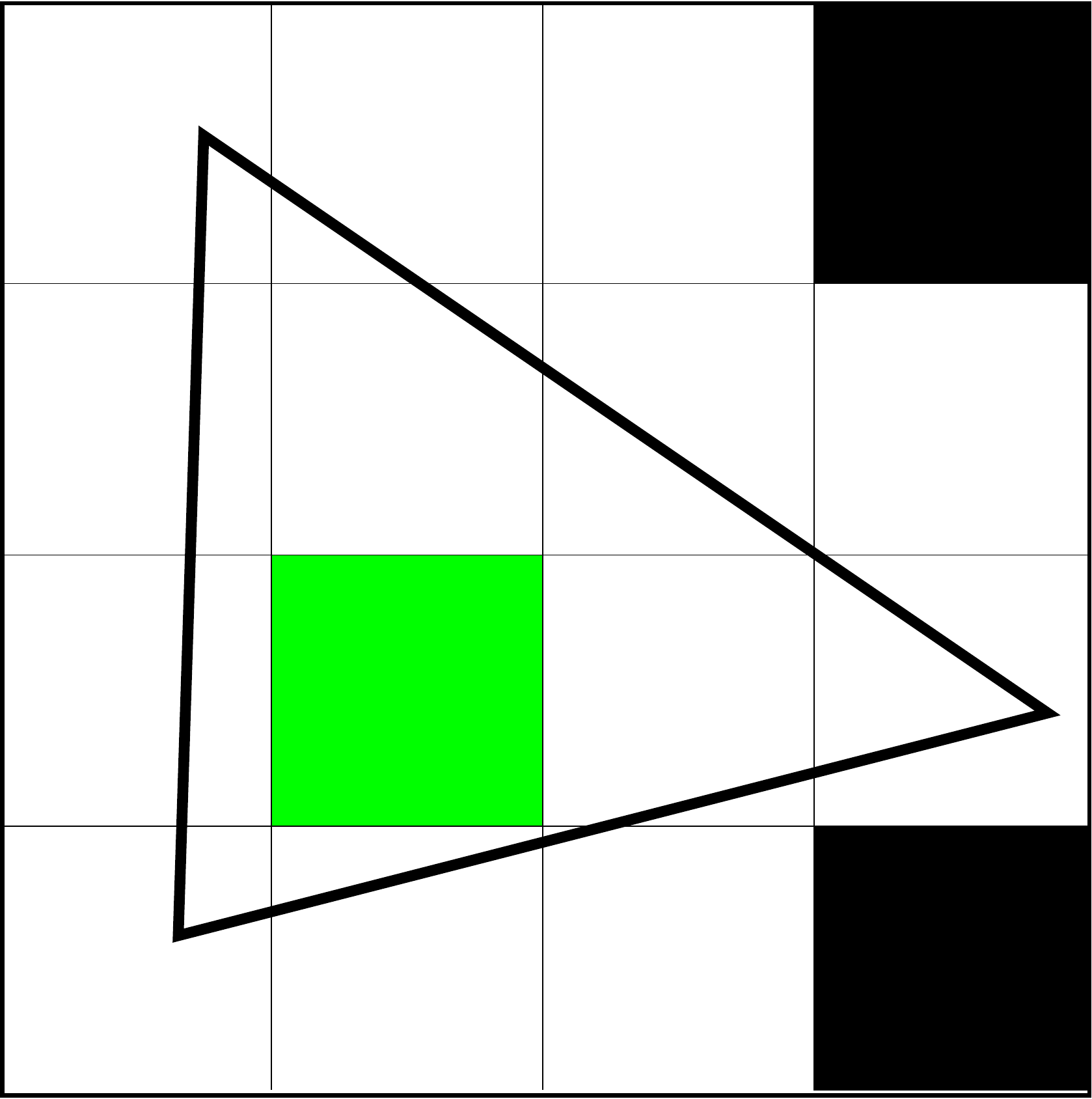}
\includegraphics[width=3cm, page=10]{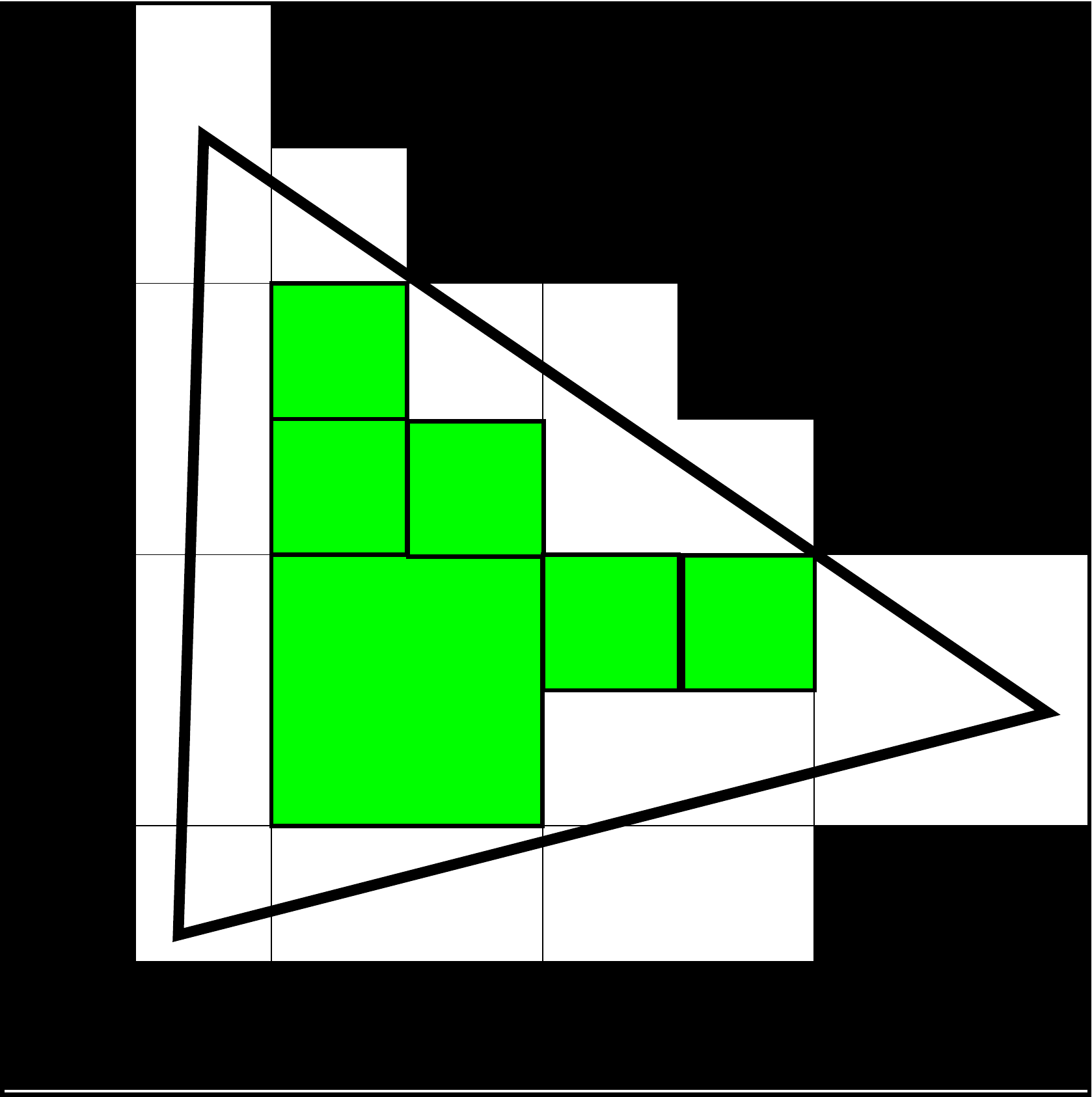}
\includegraphics[width=3cm, page=11]{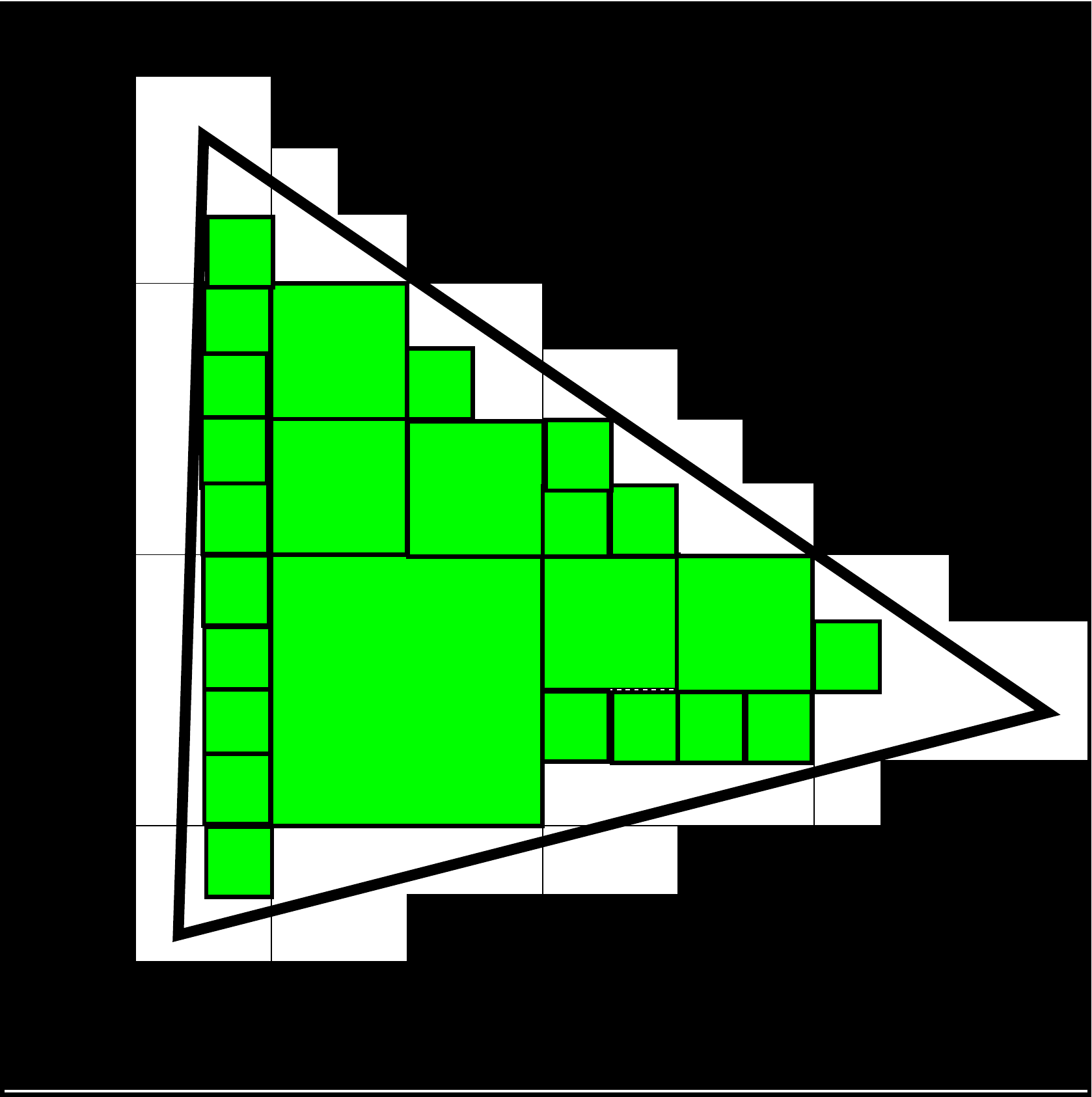}
\caption{An illustration of the covering algorithm which is encoded in the function $\Cover_v$: We essentially use a greedy algorithm and distinguish between ``good" (green), remaining (white) and discarded (black) parallelograms.}
\label{fig:cover}
\end{figure}

\section{Numerical Implementation of the Convex Integration Schemes for the Hexagonal-to-Rhombic Phase Transformation}
\label{sec:numerics}

In this section, we present some of the output of Algorithm \ref{alg:conv_int} in both of the variants described above. Here we use a colour coding based on the CMYK colour model, where the colours cyan, magenta and yellow (we do not use key) correspond to one of the wells respectively. More precisely, the well $e^{(1)}$ corresponds to cyan, the well $e^{(2)}$ to magenta, and $e^{(3)}$ to yellow.
A matrix which is included in the interior of the convex hull is correspondingly
depicted as a convex combination of these colourings (c.f. Figure
\ref{fig:repl_colours}).

We remark that in commercial printing the CMYK colour model is used to blend colours as mixtures of the
primary colours cyan, magenta, yellow (and black). In particular,
for our purposes this is very convenient, as it allows us to work with ``barycentric coordinates" based on cyan, magenta and yellow. 
Moreover, too fine
structures in our construction, which are hard to see in the picture due to limited
resolution and rasterisation, are then ``homogenised" automatically in a way which agrees with the colours of the averaged matrix values. 

\subsection{Output of the implementation of Algorithm \ref{alg:conv_int}}

We compare the output of the two variants of Algorithm \ref{alg:conv_int} for three different scenarios. These are chosen such that the overall boundary data $M \in \intconv(\{e^{(1)}, e^{(2)},e^{(3)}\})$ are such that either (c.f. Figure \ref{fig:location})
\begin{itemize}
\item[(i)] $M$ is close to the barycenter $\frac{1}{3}e^{(1)}+\frac{1}{3}e^{(2)}+\frac{1}{3}e^{(3)}=0$ of the equilateral triangle spanned by $e^{(1)}, e^{(2)}, e^{(3)}$.
\item[(ii)] $M$ is close to one of the wells.
\item[(iii)] $M$ is close to the boundary of $\intconv(e^{(1)},e^{(2)}, e^{(3)})$ and essentially in between two wells.
\end{itemize}
We remark that, due to the $Skew(2)$ symmetry, our construction in Algorithm 3.3
essentially only depends on the symmetric part $e(M)$ of the matrix $M$ (the
skew part only amounts to orthogonal translations of the matrices). We discuss the model cases (i)-(iii) separately in the sequel.

\begin{figure}[t]
\includegraphics[width=4cm, page=12]{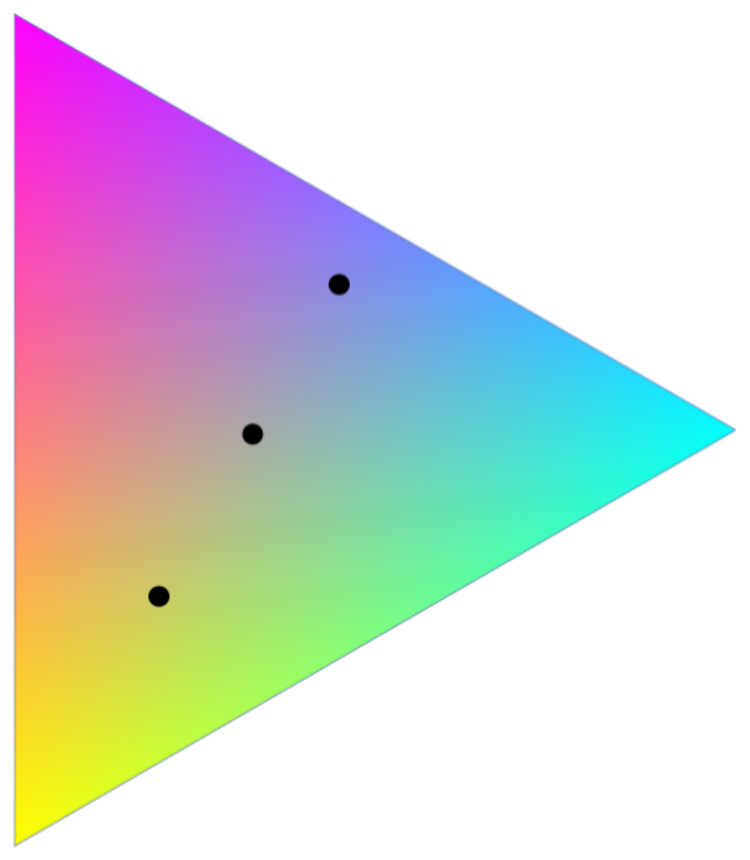}
\caption{The three cases (i)-(iii). The matrices marked in the triangle show the location of the boundary data from cases (i)-(iii) in the equilateral triangle spanned by $e^{(1)}, e^{(2)}, e^{(3)}$.}
\label{fig:location}
\end{figure}

\subsubsection{An example of case (i)}
As a first case, we discuss boundary data which are chosen to lie very close to the barycenter of the equilateral triangle spanned by the matrices $e^{(1)}, e^{(2)}, e^{(3)}$. 

\begin{figure}[t]
\includegraphics[height=6 cm, page=13]{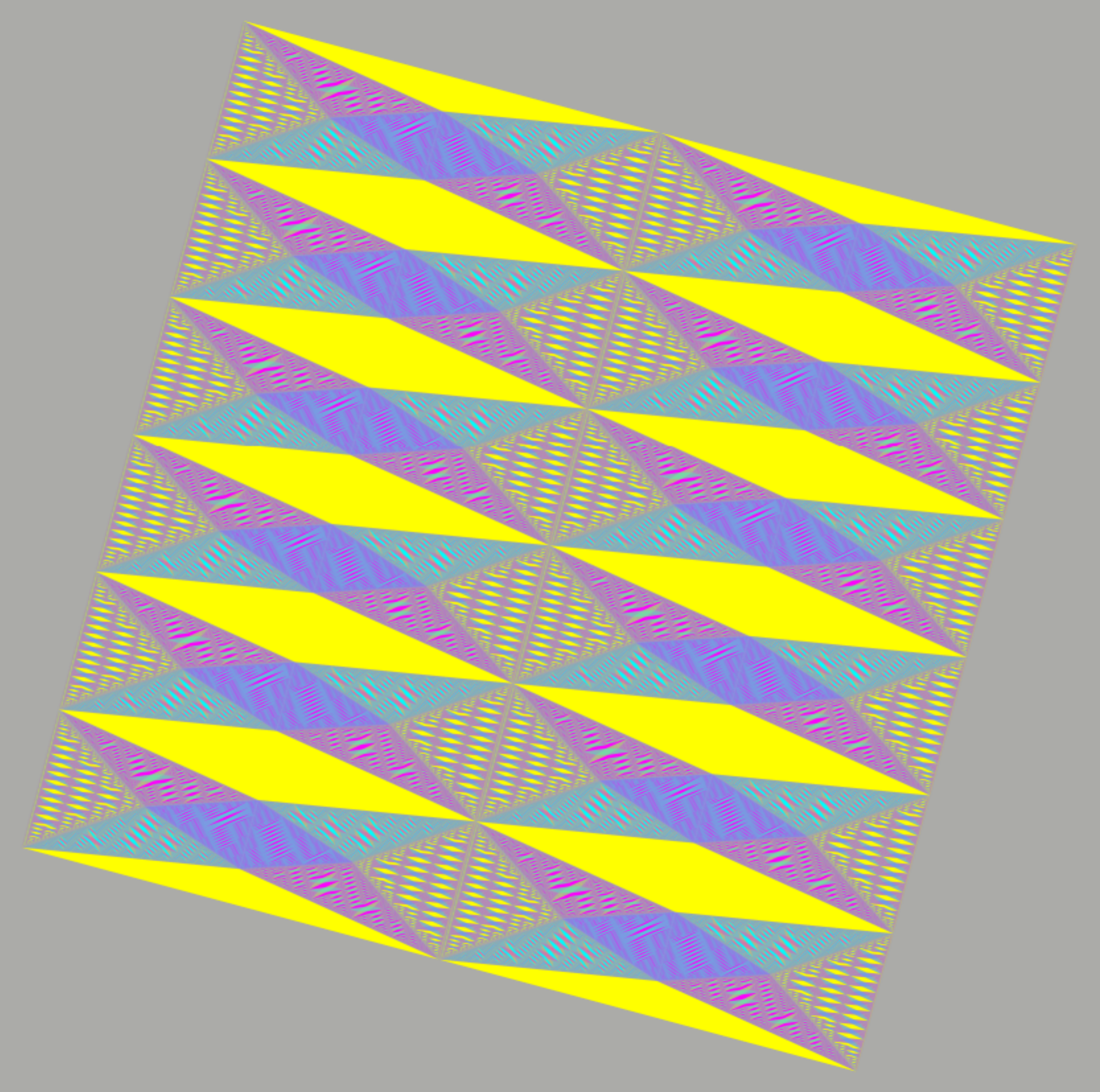}
\includegraphics[height=6 cm, page=14]{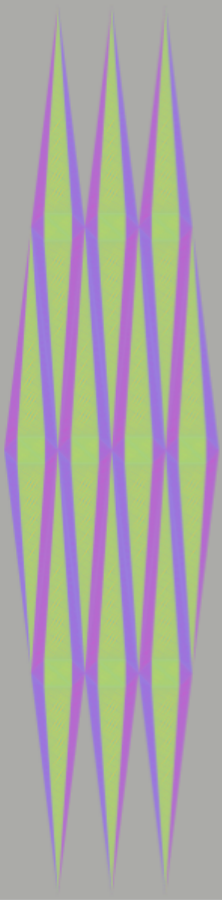}
\caption{Comparison of the algorithms from \cite{RZZ16}, \cite{RZZ17} in the case (i): Both pictures illustrate the solution of Algorithm \ref{alg:conv_int} with the choice of boundary data of the form
$M= 0.33 e^{(1)} + 0.33 e^{(2)} + 0.34 e^{(3)}.$
Hence, the boundary data are a quite uniform mixture of all three wells. In the
square construction, the largest patches however correspond to the majority
phase, i.e. yellow, while the diamond construction displays a more ``fine scale"
structure (c.f. Figure \ref{fig:zoomed} for a detailed view). Since the boundary values are close to the barycenter of the equilateral triangle spanned by $e^{(1)}, e^{(2)}, e^{(3)}$, the diamond construction does not yet display the effects of the strong geometric convergence, which it has for data closer to the wells.}
\label{fig:case_i}
\end{figure}

\begin{figure}[h]
\includegraphics[width=0.4\linewidth, page=15]{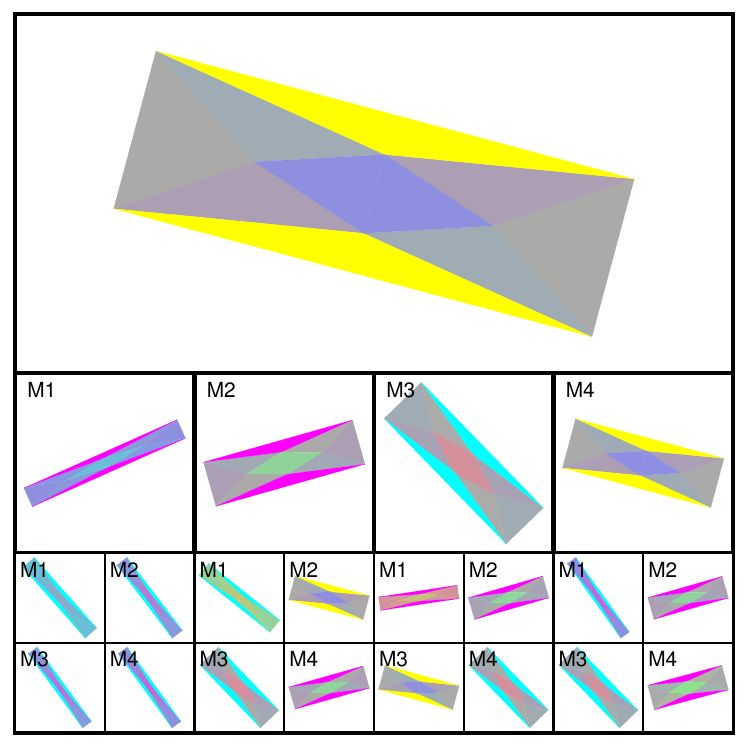}
\includegraphics[width=0.4\linewidth, page=16]{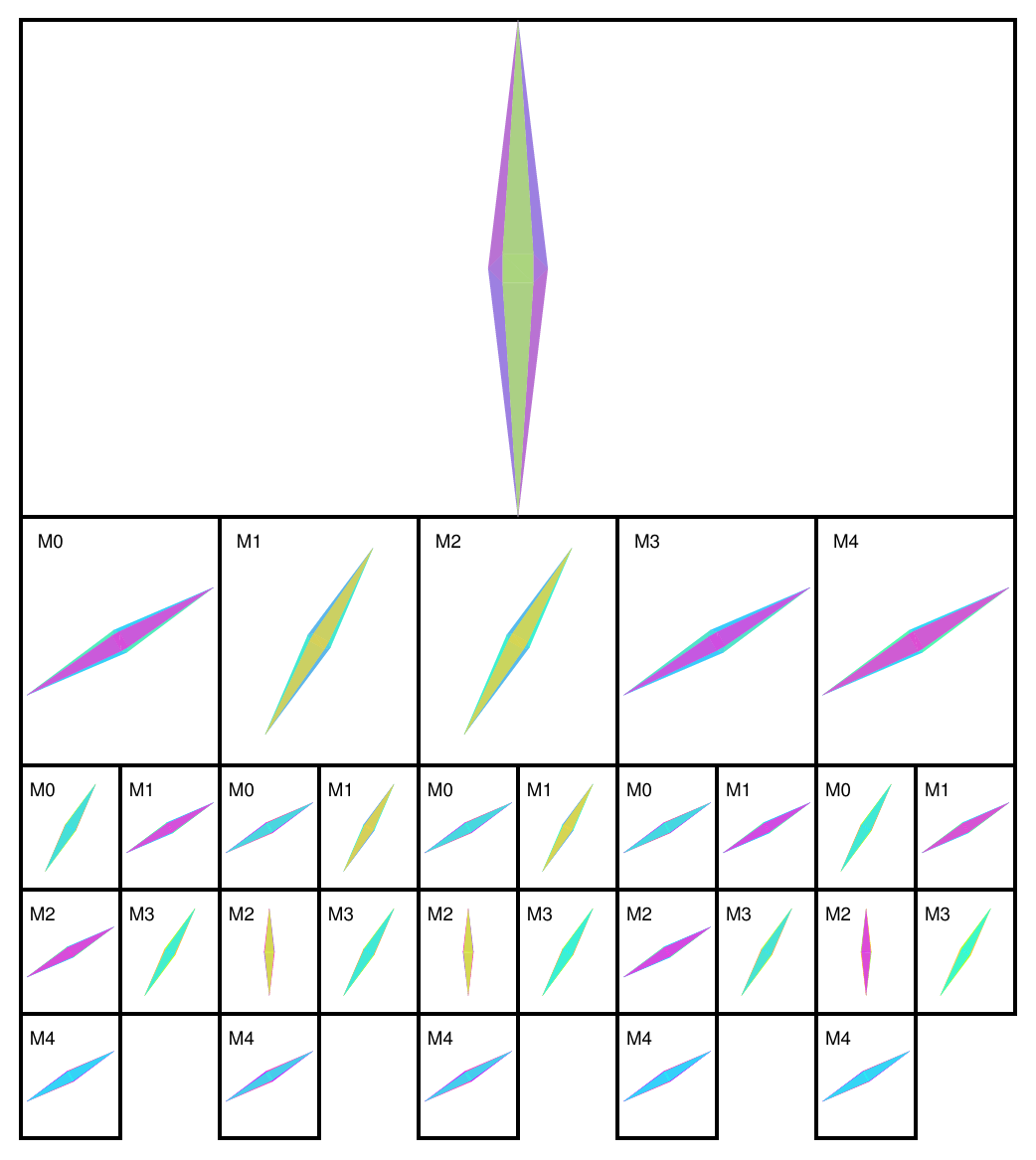}
\caption{The replacement constructions which are used in the algorithm leading
  to the constructions in Figure \ref{fig:case_i}. The left panel shows the
  iteration of the replacement constructions, which are used in the left panel of
  Figure \ref{fig:case_i} (corresponding to the algorithm from \cite{RZZ16}), while the right panel shows the iterative replacement
  constructions used in the iteration leading to the right panel of Figure \ref{fig:case_i}, i.e. the algorithm from \cite{RZZ17}. In both situation the configurations replacing the previous ones are arranged below each other. The labelling $M_0, M_1, \dots, M_4$ indicates which of the matrices in the previous step is replaced by the corresponding new replacement construction. Although both settings involve five different level sets in each replacement construction, we have chosen to plot only four of these in the left panel, as the fifth corresponds to one of the wells, which is not changed at later iteration steps and can be easily inferred from the previous step. Although it
  is not always easy to detect this by the ``eye norm", the constructions which are used in the left panel are slightly rotated with respect to each other (e.g. the different replacement constructions involving yellow), while there are only a discrete number of orientations in the left panel.}
\label{fig:case_i1}
\end{figure}

In the implementations of the two algorithms from \cite{RZZ16} and \cite{RZZ17} we see major differences, which are illustrated in Figure \ref{fig:case_i}. As discussed in Section \ref{sec:replace} we observe the following differences:
\begin{itemize}
\item finiteness: The construction of \cite{RZZ16} is finite, which is
  manifested in the fact that already after three iteration steps there are
  large patches (approximately $20$ percent), in which the final colouring is attained (in Figure \ref{fig:case_i} for instance there are large yellow patches). In contrast to this, the algorithm from \cite{RZZ17} is countably infinite and although in the fine structure one already sees colours close to the wells, these are not exactly in the wells in general. In particular this yields a very fine-scale structure, which  ``homogenises" in the ``eye norm". 
\item uniformity: In the implementation of the algorithm \cite{RZZ16} the replacement constructions always contain patches which are ``pushed away" into the interior (c.f. Figure \ref{fig:repl_colours}). In the implementation of the algorithm from \cite{RZZ17}, we in contrast see that the replacement constructions use constructions whose colours are increasingly close to the wells (c.f. Figure \ref{fig:diamond}, right panel).
\item normal directions: While the normals which are used in the algorithm from \cite{RZZ16} may vary in a continuum, the ones in the algorithm from \cite{RZZ17} are discrete. Although this is hard to see by the ``eye norm", it is manifested in the replacement constructions shown in Figure \ref{fig:case_i1}.  
\end{itemize}
Both constructions are of self-similar, fractal structure (however leading to quite different overall structures).

\begin{figure}[htbp]
  \centering
  \includegraphics[width=0.4\linewidth, page=17]{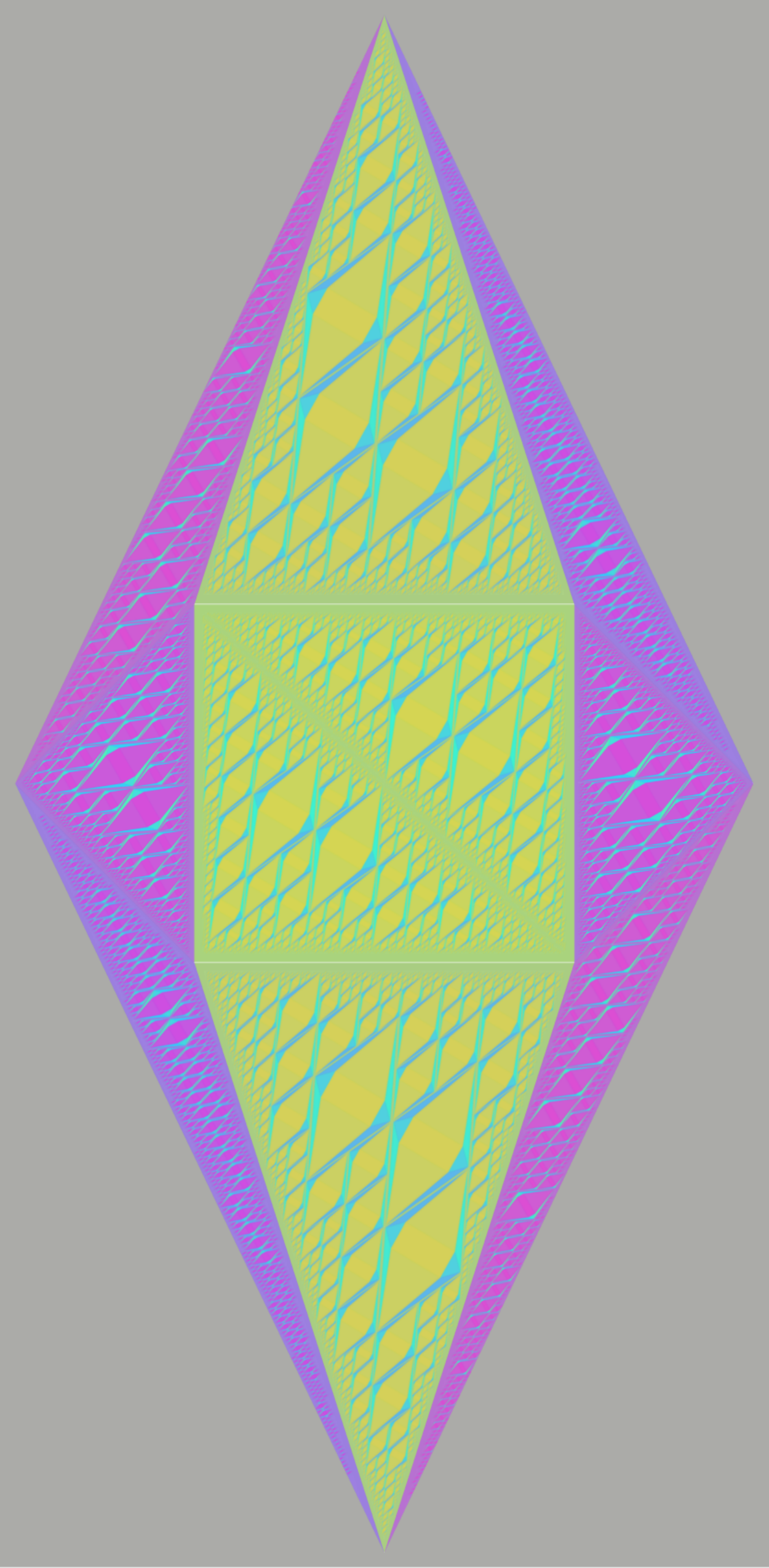}
  \includegraphics[width=0.4\linewidth, page=18]{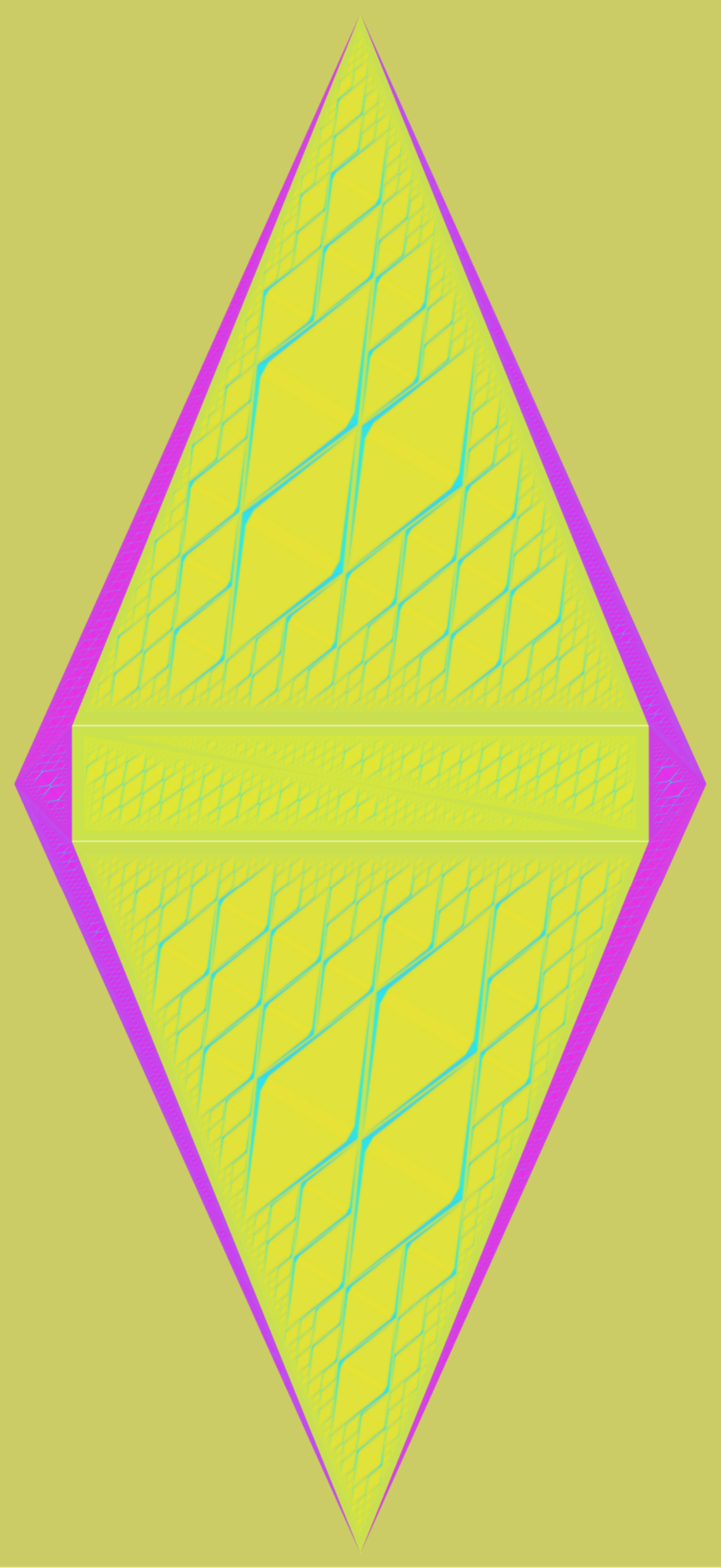}
  \includegraphics[width=0.8\linewidth, page=19]{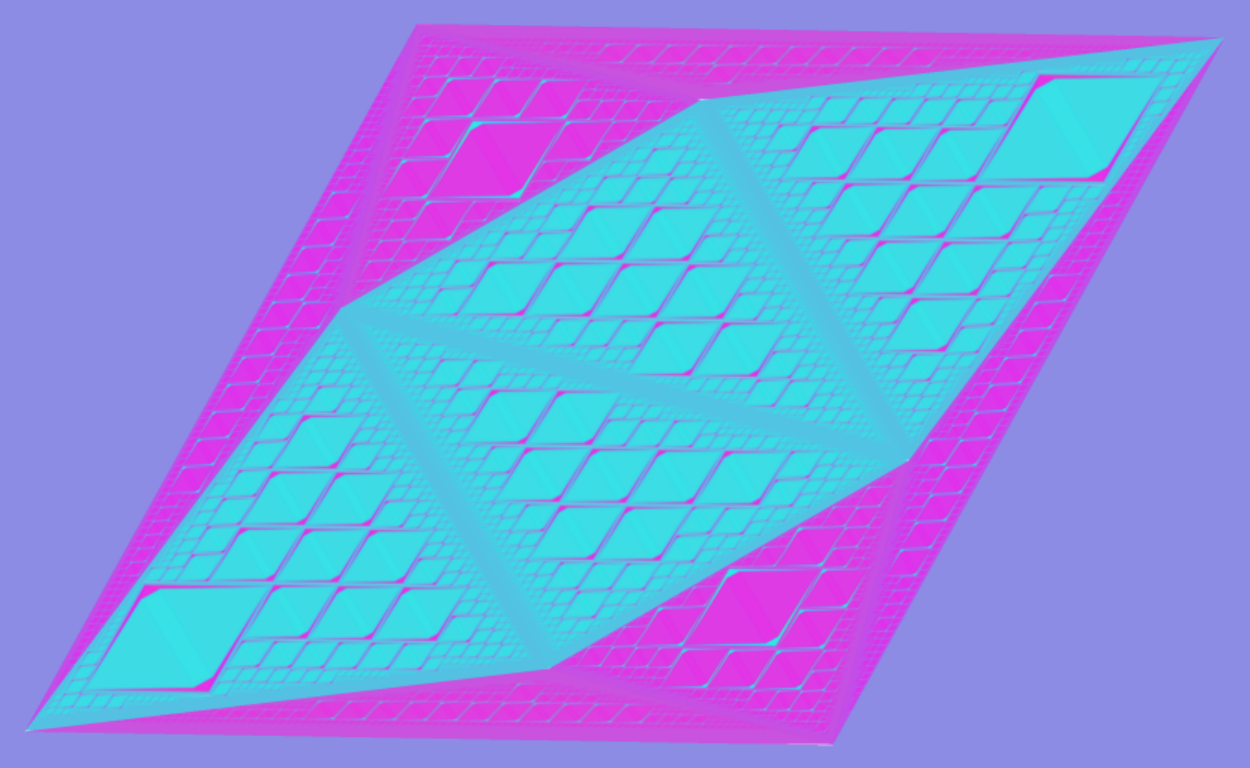}
  \caption{The construction in \cite{RZZ17} requires a small (but uniform) aspect ratio of the building block constructions in order to
    control the errors in matrix space. This results in very fine structures (c.f. Figures \ref{fig:case_i}, \ref{fig:case_ii}, \ref{fig:case_iii}, right panel).
    For illustration purposes, in the present figure we hence show the same constructions but
    with an artificially increased aspect ratio (we have roughly increased the real aspect ratio by a factor 10). From top to bottom and left to right this corresponds to the cases (i), (ii), (iii) discussed below.}
  \label{fig:zoomed}
\end{figure}

\subsubsection{An example of case (ii)}

\begin{figure}[h]
\includegraphics[height=6cm,page=20]{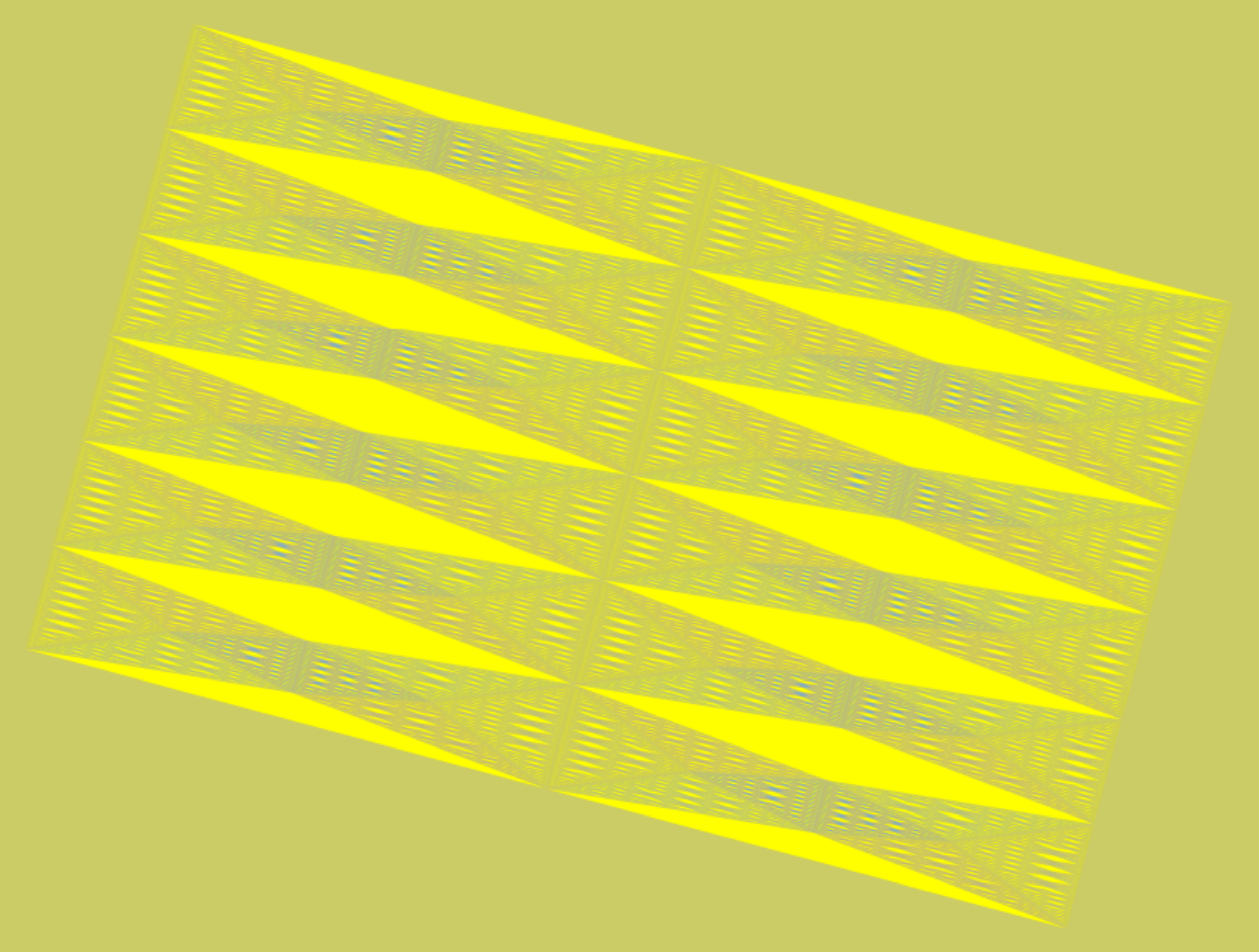}
\includegraphics[height=6cm,page=21]{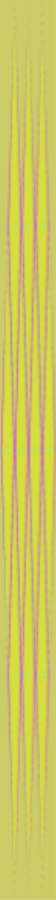}
\caption{Comparison of the algorithms from \cite{RZZ16}, \cite{RZZ17} in the
  case (ii): In this illustration of the algorithm we have chosen the boundary
  data to be $M=0.2 e^{(1)} +0.2 e^{(2)} + 0.6 e^{(3)}$. One of the most
  striking differences after three iteration steps between the algorithms from
  \cite{RZZ16} (left panel) and from \cite{RZZ17} (right panel), is the fact
  that the first algorithm still essentially stays in a neighbourhood of a
  single well (only in the third iteration step a very small fraction of cyan emerges on small scales). In contrast, the algorithm from
  \cite{RZZ17} always uses two wells in each iteration step. In particular, one clearly sees the emergence of magenta on the outer structures of the diamond and after zooming in also cyan can be recognized (c.f. Figure \ref{fig:zoomed}, second panel, for a detailed view).
}
\label{fig:case_ii}
\end{figure}

In addition to the differences which had already been discussed in the case (i),
the most striking point which can is reflected in the implementations of the case
(ii) is the ``inertia" which is underlying the algorithm from \cite{RZZ16}
before the well is changed. This can best be understood by recalling the
replacement algorithm from Section \ref{sec:replace}, which is also illustrated
in Figure \ref{fig:repl_colours}. Here the original boundary data are replaced
by five new matrices, of which some are ``pushed further into the interior" of
the convex hull of the strains $e^{(1)}, e^{(2)}, e^{(3)}$. This implies that
the closest well eventually changes. However, since the push-out is achieved
with a fixed factor, this is slower, the closer the original boundary datum had
been to one of the wells. In particular, if one starts with a matrix very close to one of the wells, it might take a large number of iteration steps before any other well is used (as the algorithm will have reached very small scales by then, it is also hard to see this, even if this happens).
   
In contrast to this, the algorithm from \cite{RZZ17} uses values close to two of the wells in each iteration step: Hence even if one starts close to a single well, the other wells will emerge within the next two iteration steps (c.f. the second panel of Figure \ref{fig:zoomed}).

Again both of the constructions (the one from \cite{RZZ16} and the one from \cite{RZZ17}) already display fractal behaviour.

\subsubsection{An example of case (iii)}

Last but not least, we discuss the setting in which the boundary data are very close to the rank-one line between two wells. Here both constructions mainly consist of shades of cyan and magenta. The angle between the closest well and the boundary is very flat. Hence, we expect that the construction from \cite{RZZ16} (left panel in Figure \ref{fig:case_iii}) is forced to use very fine scales depending on the boundary distance. In particular, the regularity in the finely twinned areas should be bad, eventually yielding regularities which are significantly rougher than the ones from cases (i), (ii) (c.f. Figures \ref{fig:case_i}, \ref{fig:case_ii}). In contrast, the construction from \cite{RZZ17} (right panel in Figure \ref{fig:case_iii}) is not significantly more complex than the constructions from Figures \ref{fig:case_i}, \ref{fig:case_ii}.

\begin{figure}[t]
\includegraphics[height=5cm, page=22]{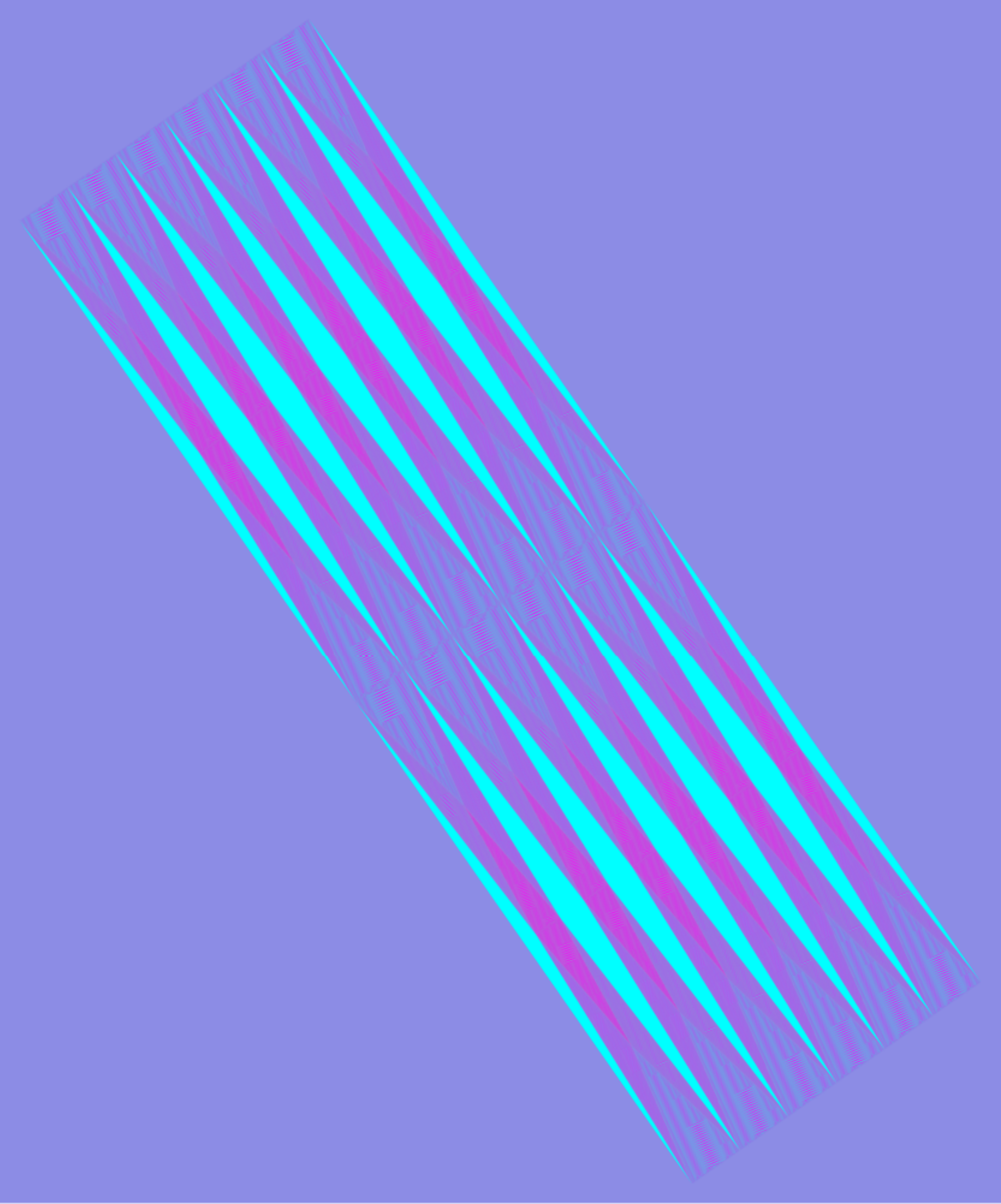}
\includegraphics[height=5cm, page=23]{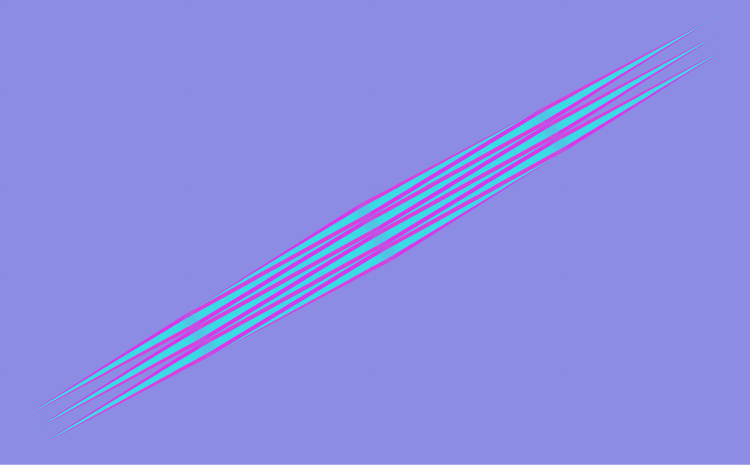}
\caption{
An implementation of the convex integration algorithm for boundary data which are very close to the boundary of the convex hull of the wells. The boundary data are chosen to be
$M=0.45 e^{(1)} + 0.45 e^{(2)} + 0.1 e^{(3)}$. The left panel shows the
implementation of the algorithm from \cite{RZZ16}, the right panel the one from
\cite{RZZ17}. The complexity of the right panel is comparable to the
constructions from cases (i), (ii). Although this is not easy to detect by the
eye norm, the complexity of the figure in the left panel is significantly more
complex than the ones from these cases: The scales replacing the old structures
are so fine, that they are homogenized by rasterisation in pictures of algorithm and are
hence hardly detectable (c.f. Figure \ref{fig:zoomed} for a detailed view).
}
\label{fig:case_iii}
\end{figure}

\section{Discussion and Summary}
\label{sec:discuss}

We conclude the article by comparing our findings with some related models. Here we in particular return to their relation to scaling laws and comment on their relation to the Ball-Planes model \cite{BCH15,TIVP17}.

\subsection{Relation to scaling limits}
As already discussed in the introduction, scaling laws for energies such as \eqref{eq:min_surf} play an important role in the mathematical analysis of shape-memory materials. It is not known whether energies which include surface terms and which hence penalise too high oscillations are compatible with convex integration solutions. In our Theorem \ref{thm:lower_bd} we could however relate the presence of lower bounds to regularity thresholds for convex integration solutions. In this sense, models involving surface energies could be interpreted as a \emph{selection mechanism}.

\subsection{Probabilistic models on nucleation}

An interesting, probabilistic model for nucleation displaying ``wild" structures has been proposed by the groups around J. Ball and A. Planes and E. Vives \cite{BCH15, TIVP17}. This model takes the dynamic process of nucleation into account and assumes that in the nucleation process the different phases are nucleated according to the following rules (there are different subschemes of this): Let $\Omega \subset \R^2$ be a bounded Lipschitz domain representing the sample. The model is then based on the following algorithm:
\begin{itemize}
\item Choose a random point $x\in \Omega$.
\item According to a prescribed probability distribution, choose from a set of finitely many possible orientations, which also each have an associated colour. 
Physically, these essentially model the rank-one lines between the wells and hence also the variants of martensite associated with the given orientation. 
\item ``Nucleate" the chosen variant of martensite with the prescribed orientation at the point $x\in \Omega$. This is achieved by inserting a line segment (with or without prescribed width) with a prescribed associated colour (for each orientation) through the point $x\in \Omega$ of the prescribed orientation. This line segment is maximally extended through the point $x\in \Omega$ with the chosen orientation until it hits the boundary of $\Omega$ or another already present line, in which case it terminates.
\end{itemize}
It is assumed that the transformation process is irreversible, i.e. an already inserted line cannot be erased in the explained algorithm.

This seemingly ``simple" process is used to explain the lack of a fixed length scale, the experimentally observed fractal behaviour and the intermittency in the nucleation process of martensite \cite{CMOPV98,OCGVMP95,GMR10}. A certain ``universality" of the distribution of the lengths of line segments is identified \cite{BCH15, TIVP17}. However, compatibility considerations do \emph{not} enter the model.\\

Our geometrically more complicated setting is not viewed as a dynamic process (although it might be possible to give a dynamic meaning to the iteration parameter $k \in \N$ in the convex integration algorithm). However, in contrast to the Ball-Planes-Vives model, we include compatibility in our considerations. The lack of a fixed length scale and the fractal behaviour are also reflected in our convex integration solutions. We also expect fat tailed distributions for the Fourier transforms of our length scale distribution functions, c.f. \cite{PLKK97}. Mathematically this and the universality can be viewed as an (ir)regularity result. It would be interesting to see whether there is a closer relation between the two models and whether the two models could be combined.

We note that both models produce highly non-unique solutions. Yet, it might be possible to identify \emph{statistical} properties characterising these (possibly uniquely).

\subsection{Experiments}
It is not known whether convex integration solutions really appear in nature. Although, there are indications that in for instance the cubic-to-monoclinic phase transformation, the nucleation of martensite in austenite creates quite wild microstructures \cite{I}, it would be interesting to further compare the experimental results with numerical/analytical predictions.

\subsection{Summary}
In the present note we have presented two main results: 

On the one hand, we have related the maximal possible regularity of convex integration solutions to the existence of lower bounds in the associated scaling laws. This could present a possibility of deducing bounds on the maximal regularity of convex integration solutions (at least in ``sufficiently simple" models for which one can hope to obtain scaling laws).

On the other hand, we have presented a first implementation of convex integration solutions in the context of the hexagonal-to-rhombic phase transformation, which we view as a model problem. We have compared the numerical outputs of the two algorithms from \cite{RZZ16} and \cite{RZZ17}.

Of course this is only a first modest step towards an improved understanding of convex integration solutions, which we seek to further study in future work. Although experimental results on nucleation suggest quite wild structures in certain phase transformations \cite{I}, we have no direct comparison between these solutions and numerical predictions yet (which in particular is expected to be difficult to obtain due to the intrinsic non-uniqueness of convex integration solutions; a possible remedy could be the \emph{statistical} analysis of these solutions). We seek to pursue this comparison in future work. In particular, we emphasise that the question whether convex integration solutions arise in nature in the context of phase transformations in shape-memory alloys remains an exciting open problem.

\section{Appendix}
\label{sec:append}

In this last section, for self-containedness, we recall a possible proof of the fractional Poincar\'e inequality (for fractional Besov type spaces), which is used in different forms in Section \ref{sec:rigidity_and_scaling}. We start with the $L^2$ based version:

\begin{lem}
\label{lem:frac_poinc_L2}
Let $s\in (0,1)$ and $n\geq 1$.
Let $\Omega \subset \R^n$ be an open, bounded $C^{1,1}$ domain.
Let $u\in H^{s}(\R^n)$ with  $\supp(u) \subset \overline{\Omega}$ and denote $\Omega_{\delta}=\{x\in \Omega: \dist(x,\partial \Omega)\geq \delta\}$ for $\delta>0$.
Then, there exists $C=C(s,\Omega,n)>1$ such that
\begin{align}
\label{eq:L2Poinc}
\|u\|_{L^2(\Omega_{\delta}\setminus \Omega_{2\delta})}
\leq C \delta^{s} \|u\|_{H^{s}(\R^n)}.
\end{align}

\end{lem}

\begin{proof}
We use the Caffarelli-Silvestre extension $\overline{u}$ of $u$ \cite{CS07}, which is the $H^{1}(\R^{n+1}_+, x_{n+1}^{1-2s})$ solution to the equation
\begin{align*}
\nabla \cdot x_{n+1}^{1-2s} \nabla \overline{u} & = 0 \mbox{ in } \R^{n+1}_+,\\
\overline{u} & = u \mbox{ on } \R^n \times \{0\}.
\end{align*}
For $x_1 \in \Omega_{\delta} \setminus \Omega_{2\delta}$,
$x_2=p(x_1) \in \R^n \setminus \overline{\Omega}$ with $|x_1-x_2|\leq 4 \delta$ and $\tilde{\delta} \in [\delta, 2\delta]$, the fundamental theorem and Hölder's inequality yield
\begin{align*}
|u(x_1)|^2 &= |\overline{u}(x_1,0)|^2
\leq 2|\overline{u}(x_1,\tilde{\delta})|^2 + 
2 \left( \int\limits_{0}^{\tilde{\delta}}|\p_t \overline{u}(x_1,t)| dt \right)^2\\
&\leq 2|\overline{u}(x_1,\tilde{\delta})|^2 + 
C \tilde{\delta}^{2s}\int\limits_{0}^{\tilde{\delta}} t^{1-2s}|\p_t \overline{u}(x_1,t)|^2 dt\\
& \leq 2 |\overline{u}(p(x_1),\tilde{\delta})|^2 + C \tilde{\delta}^2 \int\limits_{0}^{1}|\nabla' \overline{u}(\tau x_1 + (1-\tau)p(x_1), \tilde{\delta})|^2 d\tau\\
& \quad + C \tilde{\delta}^{2s}\int\limits_{0}^{\tilde{\delta}} t^{1-2s}|\p_t \overline{u}(x_1,t)|^2 dt\\
& \leq  2\tilde{\delta}^2 \int\limits_{0}^{1}|\nabla' \overline{u}(\tau x_1 + (1-\tau)p(x_1), \tilde{\delta})|^2 d\tau + C \tilde{\delta}^{2s}\int\limits_{0}^{\tilde{\delta}} t^{1-2s}|\p_t \overline{u}(x_1,t)|^2 dt\\
& \quad
+ C \tilde{\delta}^{2s}\int\limits_{0}^{\tilde{\delta}} t^{1-2s}|\p_t \overline{u}(p(x_1),t)|^2 dt.
\end{align*}
Here the constant $C>1$ changes from line to line and depends on $s$ but not on $\delta$. In deriving the above estimate, we applied the fundamental theorem thrice: First in the normal direction (where we then used Hölder's inequality to insert the weight $t^{1-2s}$), then in the tangential directions and finally once more in the normal direction (where we used the vanishing Dirichlet data for $u(p(x_1))=\overline{u}(p(x_1),0)$).
Estimating the integrals involving the normal derivative by using $\delta <\tilde{\delta}\leq 2\delta$, we thus infer
\begin{align*}
|u(x_1)|^2 
& \leq  2\tilde{\delta}^2 \int\limits_{0}^{1}|\nabla' \overline{u}(\tau x_1 + (1-\tau)p(x_1), \tilde{\delta})|^2 d\tau + C \delta^{2s}\int\limits_{0}^{2\delta} t^{1-2s}|\p_t \overline{u}(x_1,t)|^2 dt\\
& \quad
+ C \delta^{2s}\int\limits_{0}^{2\delta} t^{1-2s}|\p_t \overline{u}(p(x_1),t)|^2 dt.
\end{align*}
Averaging over $\tilde{\delta} \in [\delta, 2\delta]$ yields
\begin{align*}
|u(x_1)|^2 
& \leq 2\delta^{-1} \int\limits_{\delta}^{2\delta} \tilde{\delta}^2 \int\limits_{0}^{1}|\nabla' \overline{u}(\tau x_1 + (1-\tau)p(x_1), \tilde{\delta})|^2 d\tau d \tilde{\delta}
+ C \delta^{2s}\int\limits_{0}^{2\delta} t^{1-2s}|\p_t \overline{u}(x_1,t)|^2 dt\\
& \quad
+ C \delta^{2s}\int\limits_{0}^{2\delta} t^{1-2s}|\p_t \overline{u}(p(x_1),t)|^2 dt\\
& =  2\delta^{-1} \int\limits_{\delta}^{2\delta} \tilde{\delta}^{1+2s} \tilde{\delta}^{1-2s} \int\limits_{0}^{1}|\nabla' \overline{u}(\tau x_1 + (1-\tau)p(x_1), \tilde{\delta})|^2 d\tau d \tilde{\delta}
+ C \delta^{2s}\int\limits_{0}^{2\delta} t^{1-2s}|\p_t \overline{u}(x_1,t)|^2 dt\\
& \quad
+ C \delta^{2s}\int\limits_{0}^{2\delta} t^{1-2s}|\p_t \overline{u}(p(x_1),t)|^2 dt\\
& \leq  C \delta^{-1} \delta^{1+2s} \int\limits_{\delta}^{2\delta} \tilde{\delta}^{1-2s} \int\limits_{0}^{1}|\nabla' \overline{u}(\tau x_1 + (1-\tau)p(x_1), \tilde{\delta})|^2 d\tau d \tilde{\delta}
+ C \delta^{2s}\int\limits_{0}^{2\delta} t^{1-2s}|\p_t \overline{u}(x_1,t)|^2 dt\\
& \quad
+ C \delta^{2s}\int\limits_{0}^{2\delta} t^{1-2s}|\p_t \overline{u}(p(x_1),t)|^2 dt\\
& \leq  C \delta^{2s} \int\limits_{0}^{2\delta} \int\limits_{0}^{1} t^{1-2s} |\nabla' \overline{u}(\tau x_1 + (1-\tau)p(x_1), t)|^2 d\tau dt
+ C \delta^{2s}\int\limits_{0}^{2\delta} t^{1-2s}|\p_t \overline{u}(x_1,t)|^2 dt\\
& \quad
+ C \delta^{2s}\int\limits_{0}^{2\delta} t^{1-2s}|\p_t \overline{u}(p(x_1),t)|^2 dt.
\end{align*}
Finally, integrating over $x_1 \in \Omega_{\delta} \setminus \Omega_{2\delta}$ leads to
\begin{align*}
\|u\|_{L^2(\Omega_{\delta}\setminus \Omega_{2\delta})}^2
 &\leq C \delta^{2s} \int\limits_{0}^1 \int\limits_{0}^{2\delta} t^{1-2s} \|\nabla' \overline{u}(\tau \cdot + (1-\tau)p(\cdot),t)\|_{L^2(\Omega_{\delta}\setminus \Omega_{2\delta})}^2 dt d\tau\\
& \quad + 2C \delta^{2s} \|t^{\frac{1-2s}{2}} \p_t \overline{u}\|_{L^2((\Omega_{\delta} \setminus \Omega_{2\delta}) \times [0,2\delta])}^2\\
& \leq C \delta^{2s} \|t^{\frac{1-2s}{2}} \nabla \overline{u}\|_{L^2(\Omega \times [0,2\delta])}^2\\
& \leq C \delta^{2s}  \|t^{\frac{1-2s}{2}} \nabla \overline{u}\|_{L^2(\R^{n+1}_+)}^2.
\end{align*}
Hence, using that (c.f. \cite{CS07})
\begin{align*}  
\|t^{\frac{1-2s}{2}} \nabla \overline{u}\|_{L^2(\R^{n+1}_+)}^2
\leq C \|u\|_{\dot{H}^s(\R^n)},
\end{align*}
we obtain the desired estimate \eqref{eq:L2Poinc}.
\end{proof}

Similarly, we have the analogue of this in Sobolev and Besov spaces:

\begin{lem}
\label{lem:frac_poinc}
Let $s\in (0,1)$, $p\in (1,\infty)$ and $n\geq 1$.
Let $\Omega \subset \R^n$ be an open, bounded $C^{1,1}$ domain.
Let $u\in W^{s,p}(\R^n)$ with  $\supp(u) \subset \overline{\Omega}$ and denote $\Omega_{\delta}=\{x\in \Omega: \dist(x,\partial \Omega)\geq \delta\}$ for $\delta>0$.
Then, there exists $C=C(s,p,\Omega,n)>1$ such that
\begin{align}
\label{eq:Wsp_Poinc}
\|u\|_{L^p(\Omega_{\delta}\setminus \Omega_{2\delta})}
\leq C \delta^{s} \|u\|_{W^{s,p}(\R^n)}.
\end{align}

\end{lem}

\begin{proof}
The argument follows essentially as above. However for the estimate in the normal direction, we use
\begin{align*}
|\overline{u}(x_1,0)|^p
&\leq C |\overline{u}(x_1,\delta)|^p + C\left(\int\limits_{0}^{\delta} \p_t \overline{u} dt\right)^{p}\\
& \leq C|\overline{u}(x_1,\delta)|^p + C\int\limits_{0}^{\delta} |t^{1-\frac{1}{p}-s }\p_t \overline{u}|^p dt \left(\int\limits_{0}^{\delta} t^{-1+\frac{sp}{p-1}}dt \right)^{p-1}\\
& \leq C|\overline{u}(x_1,\delta)|^p +  C\delta^{sp} \int\limits_{0}^{\delta} t^{p-1-s p}|\p_t \overline{u}|^p dt . 
\end{align*}
Combining this with estimates in the tangential directions similarly as in the proof of Lemma \ref{lem:frac_poinc}, we obtain 
\begin{align*}
\|u\|_{L^p(\Omega_{\delta}\setminus \Omega_{2\delta})}
&\leq C \delta^{s} \|t^{1-\frac{1}{p}-s}\nabla \overline{u}\|_{L^p(\R^{n+1}_+)}
\leq C \delta^{s} \|u\|_{W^{s,p}(\R^n)}.
\end{align*}
Here we used the trace characterisation of $W^{s,p}(\R^n)$, c.f. \cite[Section 10]{LS16}, where the authors rely on the characterisation from \cite{BC15}.
\end{proof}

A similar argument using the characterisation of Besov spaces from 
\cite[Section 10]{LS16} (where the authors again rely on the characterisation from \cite{BC15}) also yields a similar Poincar\'e estimate in Besov spaces.

\bibliographystyle{alpha}
\bibliography{citations1}

\end{document}